\documentclass[11pt]{amsart}

\usepackage{amsmath,amssymb,amsthm,mathtools}
\usepackage{booktabs}
\usepackage{hyperref}
\usepackage{tikz}

\newcommand{\Sym}{S}
\newcommand{\Av}{\operatorname{Av}}
\newcommand{\Sh}{\operatorname{Sh}}
\newcommand{\mex}{\operatorname{mex}}
\newcommand{\sg}{\operatorname{sg}}
\newcommand{\id}{\operatorname{id}}

\theoremstyle{plain}
\newtheorem{theorem}{Theorem}[section]
\newtheorem{proposition}[theorem]{Proposition}
\newtheorem{lemma}[theorem]{Lemma}
\newtheorem{corollary}[theorem]{Corollary}
\newtheorem{conjecture}[theorem]{Conjecture}

\theoremstyle{definition}
\newtheorem{definition}[theorem]{Definition}
\newtheorem{example}[theorem]{Example}

\theoremstyle{remark}
\newtheorem{remark}[theorem]{Remark}

\title[Permutation Avoidance Game]{A Permutation Avoidance Game with Reverse Replies and Monotone Traps}
\author{Henning Ulfarsson}
\date{\today}

\begin{document}

\begin{abstract}
  We study the impartial game PAP (``permutations avoiding patterns''), in which players take turns choosing patterns to avoid. We define a set of length $k$ patterns, $B_k$,
  and show that it is the unique minimal
  monotone-forcing subset of $\Sym_k$: every sufficiently long permutation that avoids $B_k$ is
  monotone, and every monotone-forcing subset of $\Sym_k$ must contain $B_k$.  We prove a quadratic upper bound for the
  monotone-forcing threshold, and determine the exact thresholds for $k=3,4,5,6$.  We use properties
  of the sets $B_k$ to prove that a reverse-reply strategy wins PAP on $\Sym_n$ when $k=4$
  for all $n \geq 10$; for $k=3$, the same strategy can be analysed directly.  We conjecture that it is a winning strategy for all $k$ and $n$ sufficiently large.
\end{abstract}

\maketitle

\section{Introduction}

Fix $n \geq 1$ and let $\Sym_n$ be the set of permutations of length $n$ in one-line notation.
For $\pi = \pi(1)\cdots\pi(n) \in \Sym_n$, we write $\pi^r = \pi(n)\cdots\pi(1)$ for the
\emph{reverse} and $\pi^c = (n{+}1{-}\pi(1))\cdots(n{+}1{-}\pi(n))$ for the \emph{complement}.
We say that $\pi$ \emph{contains} a pattern $p \in \Sym_k$ if there exist indices
$1 \leq i_1 < \cdots < i_k \leq n$ such that the subsequence $\pi(i_1)\cdots\pi(i_k)$ is
order-isomorphic to~$p$; otherwise $\pi$ \emph{avoids}~$p$.
For background on permutation patterns and permutation classes, see
Vatter~\cite{vatter-permutation-classes}.
Throughout the paper we study the following normal-play impartial game.

\begin{definition}
  Fix $k \geq 1$.  In the game \emph{Permutations Avoiding Patterns}, or \emph{PAP}, \emph{with pattern length $k$}, a position is a subset
  $X \subseteq \Sym_n$.  A move consists of choosing a pattern $p \in \Sym_k$ that is contained
  in at least one permutation of $X$, and replacing $X$ by the subset of permutations in $X$ that
  avoid $p$.
\end{definition}

Because a chosen length-$k$ pattern remains forbidden for the rest of the game and there are only
$k!$ such patterns, every play terminates after at most $k!$ moves.

The game can equally well be played with more general notions of pattern, such as the mesh
patterns of Br\"and\'en and Claesson~\cite{branden-claesson}, but we restrict attention to
classical patterns throughout.  Related sequence and permutation games appear already in the
monotonic sequence games of Harary, Sagan, and West~\cite{harary-sagan-west} and the later work
of Albert et al.~\cite{albert-monotonic}.  More recently, Pudwell~\cite{pudwell-erdos-szekeres}
studied a different two-player permutation game based on the Erd\H{o}s--Szekeres theorem, where the
players build a permutation one entry at a time and try to avoid long monotone subsequences.
Combinatorial games on permutations were also studied by Parton~\cite{parton}, who introduced
PermuNim, an impartial permutation-avoidance game played on an $m\times n$ board by placing points
in unused rows and columns while avoiding a fixed pattern.  Thus Parton's positions are partial
permutation matrices on a board, and the geometry of the board itself plays a central role, rather
than subsets of a fixed symmetric group~$\Sym_n$.

We are mainly interested in the starting position $X = \Sym_n$.  If a set
$F \subseteq \Sym_k$ of patterns has already been chosen, then the current position is
\[
  \Av_n(F)=\{\pi \in \Sym_n : \pi \text{ avoids every pattern in } F\}.
\]
So for fixed $k$ the game may be viewed as a game on subsets of $\Sym_k$.

We will show that for $k=3$, a simple \emph{reverse strategy} already wins for Player~II (the second
player).  For $k=4$, and more generally for large-$n$ reverse-reply questions, the key structure is
a unique minimal monotone-forcing set $B_k \subseteq \Sym_k$ of patterns, which we define below.

From the game-theoretic point of view, the main results are these.  For $k=3$, the reverse reply
wins for every $n \geq 3$.  For $k=4$, the complete picture is that the move-by-move reverse
strategy works at $n=4$, fails at $n=5,6,7,8,9$, and then works again for all $n \geq 10$.  The
failure at small $n$ splits into two different phenomena: at $n=6$ the starting position itself is
an N-position, whereas at $n=5,7,8,9$ the starting position is still a P-position but the
reverse move need not be the correct winning reply.  The $k=4$ theorem for $n \geq 10$ is proved by
combining the structural $B_k$ analysis with a finite verification.

The paper therefore has two intertwined themes.  Sections~3--5 and~7 develop the structural theory
of the sets $B_k$: a recursive characterization, minimality, upper and lower bounds for the
monotone-forcing threshold, and threshold data for small~$k$.  Sections~2 and~6 are the
game-theoretic side: Section~2 analyzes the small starting positions, while Section~6 uses the
structure of $B_k$ to prove the large-$n$ reverse-reply theorem for $k=4$.  The role of $B_k$ in
the game is precisely that it supplies both the monotone endgame and the near-monotone witness
families that make most reverse replies uniform.

For general $k$, the paper goes only as far as a conjectural extension of this picture.  The
discussion beyond $k=4$ is intended to isolate the next structural obstruction, not to claim a
fully scalable theorem.

Several results in the paper are computer-assisted.  Those computations are all exhaustive
and, with the exception of the separate $N_6$ threshold computation, are reproduced by the scripts
collected in the public repository~\cite{game-patterns-repo}, specifically in the directory
\texttt{code-permuta/}; see Section~\ref{sec:repro}.

\section{Basic PAP analysis}

\begin{definition}
  Fix $k \geq 1$.  For a PAP position $X \subseteq \Sym_n$, the set of \emph{followers}
  $\mathcal{F}(X)$ is the set of all positions reachable from~$X$ in one move.  The
  \emph{Sprague--Grundy value}~\cite{sprague,grundy} of $X$ is defined recursively by
  \[
    \sg(X)=\mex\{\sg(Y):Y\in\mathcal{F}(X)\},
  \]
  where $\mex(S)$ denotes the minimal excludant, that is, the least non-negative integer not in~$S$.  A terminal position (one with
  $\mathcal{F}(X)=\varnothing$) has $\sg(X)=0$.  Under normal play, Player~I (the next mover) wins
  from~$X$ if and only if $\sg(X)\neq 0$ and the position is called an \emph{N-position}.
  Otherwise, when $\sg(X)=0$, it is called a \emph{P-position} (previous player wins).
\end{definition}

For the starting position with the single allowed pattern length $k$, we write $\sg(\Sym_n,k)$ for
the Sprague--Grundy value of~$\Sym_n$.

\begin{proposition}
  For every $n \geq 1$, $\sg(\Sym_n,1)=1$.
\end{proposition}

\begin{proof}
  The only legal move is to choose the pattern $1$, which removes every permutation.
\end{proof}

\begin{proposition}
  For every $n \geq 1$, $\sg(\Sym_n,2)=0$.
\end{proposition}

\begin{proof}
  If $n=1$, there are no legal moves, so $\sg(\Sym_1,2)=0$.  Now assume $n \geq 2$.
  The only length-$2$ patterns are $12$ and $21$.  Choosing $12$ leaves $\Av_n(\{12\})=\{\id_n^r\}$,
  from which the only legal move is~$21$, after which the position is empty. The case of first choosing
  $21$ is symmetric.  Hence every follower
  of $\Sym_n$ has Sprague--Grundy value~$1$, and $\sg(\Sym_n,2)=\mex\{1\}=0$.
\end{proof}

By the proposition above, for $k=2$ and $n \geq 2$, Player~II wins by choosing the reverse of
Player~I's move.  The same reverse-reply idea already settles the case $k=3$.

\begin{lemma}\label{lem:k3-avoidance}
  For $n \geq 4$ we have
  \begin{align*}
    \Av_n(123,321,132,231) & = \varnothing,       \\
    \Av_n(123,321,213,312) & = \varnothing,       \\
    \Av_n(132,231,213,312) & = \{\id_n,\id_n^r\}.
  \end{align*}
\end{lemma}

\begin{proof}
  Every permutation of length at least $4$ contains a $4$-point subpattern order-isomorphic to some
  element of $\Sym_4$, so it is enough to check the $24$ permutations in $\Sym_4$.  A direct inspection shows that every element of $\Sym_4$
  contains at least one of $123,321,132,231$, and similarly at least one of $123,321,213,312$.
  Likewise, every element of $\Sym_4$ except $1234$ and $4321$ contains at least one of
  $132,231,213,312$.
\end{proof}

\begin{theorem}\label{thm:k3}
  For every $n \geq 3$, $\sg(\Sym_n,3)=0$.
\end{theorem}

\begin{proof}
  The statement is immediate for $n=3$, so assume $n \geq 4$.
  Player II replies to every move by choosing the reverse pattern.

  If Player I chooses $123$, Player II chooses $321$. If $n \geq 5$ then Player II has already won, since by the
  Erd\H{o}s--Szekeres theorem~\cite{erdos-szekeres} every permutation of length at least $5$ contains either $123$ or $321$.  If $n=4$, a move is possible, and
  by Lemma~\ref{lem:k3-avoidance}, any third move
  must be one of $132,213,231,312$, and after that Player II can again choose the reverse and finish
  the game.

  If Player I chooses $132$, Player II chooses $231$.  If Player I next chooses $123$ or $321$, then
  Player II chooses the reverse and empties the set by Lemma~\ref{lem:k3-avoidance}.  If Player I
  chooses $213$ or $312$, then Player II again chooses the reverse, and the position becomes
  $\{\id_n,\id_n^r\}$.  In that position the only legal moves are $123$ and $321$, so it is again a
  P-position.  The cases $231$, $213$, $312$, and $321$ follow by symmetry.
\end{proof}

\begin{remark}
  The computer-assisted results in Sections~2,~6, and~7 are reproduced by the focused scripts
  in the public repository~\cite{game-patterns-repo}, namely \texttt{code-permuta/}, except for the
  separate $N_6$ threshold computation; see
  Section~\ref{sec:repro}.
\end{remark}

In the next few pages, the Sprague--Grundy language is used mainly for two purposes: to record
small-$n$ data, and to separate the statement ``$\Sym_n$ is a P-position'' from the stronger
statement that the literal move-by-move reverse rule is a winning strategy. The later $k=4$ theorem for $n \geq 10$ is proved by the
reverse-reply criterion from Corollary~\ref{cor:conditional-reverse}, which establishes the stronger
strategy statement.

A natural question is whether the same reverse strategy continues to win for general~$k$.
Table~\ref{tab:sgvalues} records Sprague--Grundy values for several small
cases.  The entries for $k=1,2,3$ follow from the results above; the entries for $k=4$ were
computed by exhaustive search; see Section~\ref{sec:repro}.

\begin{table}[ht]
  \centering
  \begin{tabular}{ccccc}
    \toprule
    $n$ & $\sg(\Sym_n,1)$ & $\sg(\Sym_n,2)$ & $\sg(\Sym_n,3)$ & $\sg(\Sym_n,4)$ \\
    \midrule
    1   & 1               & 0               & 0               & 0               \\
    2   & 1               & 0               & 0               & 0               \\
    3   & 1               & 0               & 0               & 0               \\
    4   & 1               & 0               & 0               & 0               \\
    5   & 1               & 0               & 0               & 0               \\
    6   & 1               & 0               & 0               & 2               \\
    7   & 1               & 0               & 0               & 0               \\
    8   & 1               & 0               & 0               & 0               \\
    9   & 1               & 0               & 0               & 0               \\
    10  & 1               & 0               & 0               & 0               \\
    \bottomrule
  \end{tabular}
  \caption{Sprague--Grundy values for PAP on $\Sym_n$ up to $n=10$ and $k=4$.}
  \label{tab:sgvalues}
\end{table}

The picture becomes more nuanced already at $k=4$.  The exceptional value
\[
  \sg(\Sym_6,4)=2
\]
shows that the naive guess that $\Sym_n$ should always be a P-position for fixed $k$ is false.
A further computation explains this value: among the $24$ followers $\Av_6(p)$, twelve have
Sprague--Grundy value $0$, two have value $1$, and the remaining ten have value $3$; hence
\[
  \sg(\Sym_6,4)=\mex\{0,1,3\}=2.
\]
Moreover, even when the starting position is a P-position, the move-by-move reverse strategy
need not be a winning strategy.

\begin{proposition}\label{prop:k4-reverse-data}
  For $k=4$, tests from the starting position $\Sym_n$ show that the
  move-by-move reverse strategy works for $n=4$ and fails for $n=5,6,7,8,9$.
\end{proposition}

\begin{proof}
  This follows from an exhaustive computer search of the reverse-strategy game tree for
  $4 \leq n \leq 9$; see Section~\ref{sec:repro}.
\end{proof}

Example~\ref{ex:n9-reverse-fails} illustrates how the reverse strategy can fail.

\begin{example}\label{ex:n9-reverse-fails}
  At $n=9$ the reverse reply can already fail on the very next move, even though it remains legal.
  Starting from $\Sym_9$, suppose Player~I chooses $1234$ and Player~II replies with the reverse
  $4321$.  Then Player~I can choose $1324$.  The reverse response $4231$ is still legal, since
  a count gives $334$ avoiders before $4231$ is forbidden and $2$ afterwards,
  yet it lands in an N-position:
  \[
    \sg(\{1234,4321,1324,4231\})=2.
  \]
  In fact, the winning replies after $1324$ are $3412$, $3421$, and $4312$.
\end{example}

Section~6 shows that what happened in Example~\ref{ex:n9-reverse-fails} disappears once
certain witness-families stabilize (from $n \geq 9$ for non-monotone replies) and once the monotone endgame is available (from
$n \geq 10$).

Exact computation also reveals a rigid pattern in the lengths of optimal play when $k=4$.
Figure~\ref{fig:k4-optimal-lengths} shows the distributions of complete play lengths under
optimal play from $\Sym_n$ for $5 \leq n \leq 10$, with bar heights on a logarithmic scale.  The
supports of these distributions are particularly simple: at $n=5$ they are the even lengths
$10,12,\ldots,24$; at $n=6$ they are the odd lengths $7,9,\ldots,23$; at $n=7,8,9$ they are the
even lengths $4,6,\ldots,24$; and at $n=10$ they are the even lengths $2,4,\ldots,24$.  Thus
$n=6$ is exceptional not only in its Sprague--Grundy value but also in the parity pattern of
optimal game lengths.

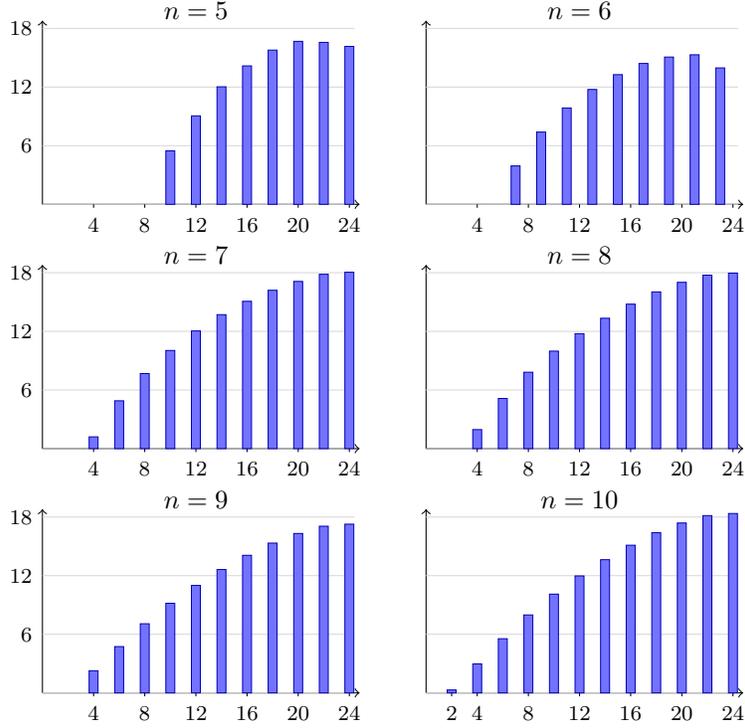
\begin{figure}[ht]
  \centering
  \begin{tikzpicture}[x=0.17cm,y=0.13cm]
    \tikzstyle{optbar}=[fill=blue!55, draw=blue!70!black, line width=0.2pt]


    \begin{scope}[shift={(0,50)}]
      \draw[->] (0,0) -- (24.8,0);
      \draw[->] (0,0) -- (0,18.8);
      \foreach \y in {0,6,12,18} {\draw[gray!30] (0,\y) -- (24.4,\y);}
      \foreach \y in {6,12,18} {\node[left,font=\scriptsize] at (0,\y) {\y};}
      \foreach \x in {4,8,12,16,20,24} {\draw (\x,0) -- (\x,-0.35) node[below,font=\scriptsize] {\x};}
      \foreach \x/\h in {10/5.473,12/9.055,14/12.027,16/14.168,18/15.783,20/16.680,22/16.574,24/16.162}
        {\filldraw[optbar] (\x-0.35,0) rectangle (\x+0.35,\h);}
      \node[font=\small] at (12,19.8) {$n=5$};
    \end{scope}

    \begin{scope}[shift={(30,50)}]
      \draw[->] (0,0) -- (24.8,0);
      \draw[->] (0,0) -- (0,18.8);
      \foreach \y in {0,6,12,18} {\draw[gray!30] (0,\y) -- (24.4,\y);}
      \foreach \x in {4,8,12,16,20,24} {\draw (\x,0) -- (\x,-0.35) node[below,font=\scriptsize] {\x};}
      \foreach \x/\h in {7/3.932,9/7.396,11/9.861,13/11.755,15/13.271,17/14.416,19/15.072,21/15.303,23/13.961}
        {\filldraw[optbar] (\x-0.35,0) rectangle (\x+0.35,\h);}
      \node[font=\small] at (12,19.8) {$n=6$};
    \end{scope}

    \begin{scope}[shift={(0,25)}]
      \draw[->] (0,0) -- (24.8,0);
      \draw[->] (0,0) -- (0,18.8);
      \foreach \y in {0,6,12,18} {\draw[gray!30] (0,\y) -- (24.4,\y);}
      \foreach \y in {6,12,18} {\node[left,font=\scriptsize] at (0,\y) {\y};}
      \foreach \x in {4,8,12,16,20,24} {\draw (\x,0) -- (\x,-0.35) node[below,font=\scriptsize] {\x};}
      \foreach \x/\h in {4/1.204,6/4.893,8/7.681,10/10.043,12/12.042,14/13.709,16/15.077,18/16.205,20/17.110,22/17.840,24/18.051}
        {\filldraw[optbar] (\x-0.35,0) rectangle (\x+0.35,\h);}
      \node[font=\small] at (12,19.8) {$n=7$};
    \end{scope}

    \begin{scope}[shift={(30,25)}]
      \draw[->] (0,0) -- (24.8,0);
      \draw[->] (0,0) -- (0,18.8);
      \foreach \y in {0,6,12,18} {\draw[gray!30] (0,\y) -- (24.4,\y);}
      \foreach \x in {4,8,12,16,20,24} {\draw (\x,0) -- (\x,-0.35) node[below,font=\scriptsize] {\x};}
      \foreach \x/\h in {4/1.944,6/5.132,8/7.817,10/9.976,12/11.763,14/13.342,16/14.784,18/16.037,20/17.017,22/17.755,24/17.960}
        {\filldraw[optbar] (\x-0.35,0) rectangle (\x+0.35,\h);}
      \node[font=\small] at (12,19.8) {$n=8$};
    \end{scope}

    \begin{scope}[shift={(0,0)}]
      \draw[->] (0,0) -- (24.8,0); 
      \draw[->] (0,0) -- (0,18.8);
      \foreach \y in {0,6,12,18} {\draw[gray!30] (0,\y) -- (24.4,\y);}
      \foreach \y in {6,12,18} {\node[left,font=\scriptsize] at (0,\y) {\y};}
      \foreach \x in {4,8,12,16,20,24} {\draw (\x,0) -- (\x,-0.35) node[below,font=\scriptsize] {\x};}
      \foreach \x/\h in {4/2.265,6/4.740,8/7.078,10/9.170,12/10.997,14/12.626,16/14.081,18/15.332,20/16.314,22/17.059,24/17.269}
        {\filldraw[optbar] (\x-0.35,0) rectangle (\x+0.35,\h);}
      \node[font=\small] at (12,19.8) {$n=9$};
    \end{scope}

    \begin{scope}[shift={(30,0)}]
      \draw[->] (0,0) -- (24.8,0); 
      \draw[->] (0,0) -- (0,18.8);
      \foreach \y in {0,6,12,18} {\draw[gray!30] (0,\y) -- (24.4,\y);}
      \foreach \x in {4,8,12,16,20,24} {\draw (\x,0) -- (\x,-0.35) node[below,font=\scriptsize] {\x};}
      \draw (2,0) -- (2,-0.35) node[below,font=\scriptsize] {2};
      \foreach \x/\h in {2/0.301,4/2.966,6/5.545,8/7.980,10/10.109,12/11.972,14/13.633,16/15.114,18/16.392,20/17.390,22/18.145,24/18.359}
        {\filldraw[optbar] (\x-0.35,0) rectangle (\x+0.35,\h);}
      \node[font=\small] at (12,19.8) {$n=10$};
    \end{scope}
  \end{tikzpicture}
  \caption{Distributions of complete play lengths under optimal play for PAP on $\Sym_n$ with
    pattern length $4$, for $5 \leq n \leq 10$.  Bar heights are $\log_{10}$ of the number of optimal
    play lines of the given length.}
  \label{fig:k4-optimal-lengths}
\end{figure}

These computations suggest that the reverse strategy does not hold uniformly in~$n$, but may still
govern the large-$n$ game.  We will prove below that for $k=4$ the reverse strategy does indeed
succeed for every $n \geq 10$.

We believe that the same phenomenon holds for general $k$: there are small-$n$ obstructions to the reverse strategy,
but for large enough $n$ it is a winning strategy for Player~II.  We therefore make the following conjecture.

\begin{conjecture}\label{conj:general-reverse}
  For every $k \geq 3$, there exists $n_0(k)$ such that the reverse strategy is a winning strategy
  for Player~II on $\Sym_n$ for all $n \geq n_0(k)$.  In particular, $\sg(\Sym_n,k)=0$ for all such
  $n$.
\end{conjecture}

To build the theoretical foundation for that later $k=4$ theorem, we first isolate the monotone
endgame and then formulate the reverse-reply criterion that we will actually use.

\begin{definition}
  We say that a set $F \subseteq \Sym_k$ is \emph{reverse-closed} if $p \in F$ if and only if
  $p^r \in F$.  For a fixed length $n$, a pattern $p$ is \emph{legal} from~$F$ at length
  $n$ if some permutation in $\Av_n(F)$ contains~$p$.  When the length $n$ is clear from the
  context, we simply say that $p$ is \emph{legal} from~$F$.
\end{definition}

\begin{lemma}\label{lem:reverse-closure-symmetry}
  If $F \subseteq \Sym_k$ is reverse-closed, then $\Av_n(F)$ is reverse-closed: $\pi \in \Av_n(F)$
  if and only if $\pi^r \in \Av_n(F)$.  In particular, a pattern $p$ is legal from $F$ if and only
  if $p^r$ is legal from~$F$.
\end{lemma}

\begin{proof}
  A permutation $\pi$ avoids $F$ if and only if $\pi^r$ avoids~$F^r = F$.
\end{proof}

\begin{proposition}\label{prop:monotone-trap}
  Let $k \geq 2$, let $F \subseteq \Sym_k$ be reverse-closed, and suppose that $F$ contains neither
  monotone pattern.  Then both monotone patterns are legal from $F$.  If
  $n \geq (k-1)^2+1$, then after one monotone pattern is chosen, the other is legal, and after both
  monotone patterns are chosen no legal moves remain.
\end{proposition}

\begin{proof}
  Since $F$ contains neither monotone pattern, the increasing permutation $\id_n$ and the decreasing
  permutation $\id_n^r$ both lie in $\Av_n(F)$.  Hence $12\cdots k$ and $k\cdots 21$ are both legal
  from $F$.

  After $12\cdots k$ is chosen, the permutation $\id_n^r$ still lies in
  $\Av_n(F \cup \{12\cdots k\})$ and contains $k\cdots 21$, so the decreasing monotone pattern
  remains legal.  The other case is symmetric.

  If both monotone patterns have been chosen and $n \geq (k-1)^2+1$, then by the Erd\H{o}s--Szekeres
  theorem every permutation of length $n$ contains at least one of them, so no permutation survives.
\end{proof}

\begin{corollary}\label{cor:conditional-reverse}
  Fix $k \geq 2$ and $n \geq (k-1)^2+1$.  Suppose that whenever $F \subseteq \Sym_k$ is reverse-closed,
  contains neither monotone pattern, and $p \in \Sym_k \setminus F$ is a non-monotone pattern legal
  from $F$, the reverse pattern $p^r$ is legal from $F \cup \{p\}$.  Then the move-by-move reverse
  strategy is a winning strategy for Player~II on $\Sym_n$.  In particular, $\Sym_n$ is a P-position.
\end{corollary}

\begin{proof}
  Starting from the empty forbidden set, Player~II always replies to Player~I by choosing the reverse
  pattern.  At each stage, let $F$ denote the current forbidden set after Player~II's previous move;
  then $F$ is reverse-closed.

  If Player~I chooses a non-monotone pattern while $F$ contains neither monotone pattern, then the
  hypothesis gives a legal reply $p^r$.

  If Player~I chooses a monotone pattern, then Proposition~\ref{prop:monotone-trap} shows that the
  other monotone pattern is legal, and that after Player~II replies there are no legal moves left.

  Thus Player~II can answer every move of Player~I by the reverse pattern, and the game ends on
  Player~II's move.  Hence the move-by-move reverse strategy is winning, and in particular
  $\Sym_n$ is a P-position.
\end{proof}

\section{Definitions and the sets \texorpdfstring{$B_k$}{Bk}}

The next three sections supply the structural input for the game result, Theorem~\ref{thm:k4}, proved in Section~6.
This section introduces the sets $B_k$ and their recursive description, Section~4 proves their
minimality via witness families, and Section~5 develops the monotone-forcing bounds on the
threshold~$N_k$.

\begin{definition}
  For $k \geq 3$, define
  \begin{align*}
    p_k & = 12\cdots (k-2)\,k\,(k-1), \\
    q_k & = 1\,k\,(k-1)\cdots 2,      \\
    r_k & = 21\,3\,4\cdots k,         \\
    s_k & = 23\cdots k\,1.
  \end{align*}
  We set
  \[
    B_k=\{p_k,q_k,r_k,s_k,p_k^c,q_k^c,r_k^c,s_k^c\}.
  \]
\end{definition}

Thus
\begin{align*}
  B_3 & = \{132,213,231,312\},                                         \\
  B_4 & = \{1243,1432,2134,2341,3214,3421,4123,4312\},                 \\
  B_5 & = \{12354,15432,21345,23451,43215,45321,51234,54312\},         \\
  B_6 & = \{123465,165432,213456,234561,543216,564321,612345,654312\}.
\end{align*}

We note that for $k=3$ we have
$p_3 = 132 = q_3$,
$r_3 = 213 = s_3^c$,
$s_3 = 231 = r_3^c$, and
$p_3^c = 312 = q_3^c$,
so $B_3$ has only four distinct elements.  For $k \geq 4$ it is easy to verify
that the eight patterns in $B_k$ are all distinct, as illustrated for $k=6$ in Figure~\ref{fig:B6-patterns}.
Notice also that
\[
  r_k=(p_k^r)^c=(p_k^c)^r
  \qquad\text{and}\qquad
  s_k=q_k^r.
\]
Thus $p_k,p_k^c,r_k,r_k^c$ form one reverse/complement orbit, and
$q_k,q_k^c,s_k,s_k^c$ form the other.

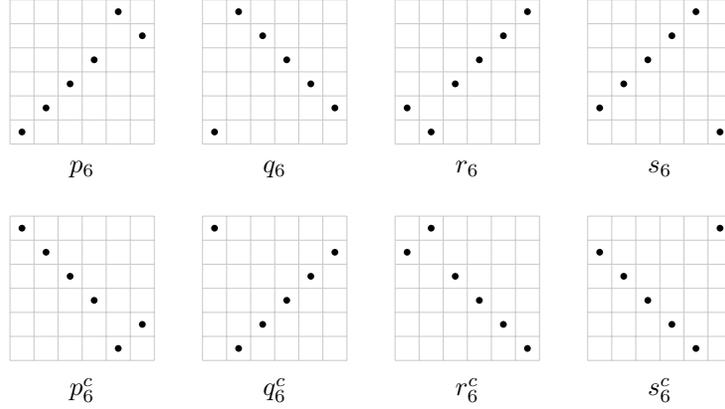
\begin{figure}[hbt]
  \centering
  \begin{tikzpicture}[scale=0.32]
    \begin{scope}[shift={(0,0)}]
      \draw[gray!40] (0,0) grid (6,6);
      \foreach \x/\y in {1/1,2/2,3/3,4/4,5/6,6/5} {\fill (\x-0.5,\y-0.5) circle (4pt);}
      \node[below] at (3,-0.3) {\small $p_6$};
    \end{scope}
    \begin{scope}[shift={(8,0)}]
      \draw[gray!40] (0,0) grid (6,6);
      \foreach \x/\y in {1/1,2/6,3/5,4/4,5/3,6/2} {\fill (\x-0.5,\y-0.5) circle (4pt);}
      \node[below] at (3,-0.3) {\small $q_6$};
    \end{scope}
    \begin{scope}[shift={(16,0)}]
      \draw[gray!40] (0,0) grid (6,6);
      \foreach \x/\y in {1/2,2/1,3/3,4/4,5/5,6/6} {\fill (\x-0.5,\y-0.5) circle (4pt);}
      \node[below] at (3,-0.3) {\small $r_6$};
    \end{scope}
    \begin{scope}[shift={(24,0)}]
      \draw[gray!40] (0,0) grid (6,6);
      \foreach \x/\y in {1/2,2/3,3/4,4/5,5/6,6/1} {\fill (\x-0.5,\y-0.5) circle (4pt);}
      \node[below] at (3,-0.3) {\small $s_6$};
    \end{scope}
    \begin{scope}[shift={(0,-9)}]
      \draw[gray!40] (0,0) grid (6,6);
      \foreach \x/\y in {1/6,2/5,3/4,4/3,5/1,6/2} {\fill (\x-0.5,\y-0.5) circle (4pt);}
      \node[below] at (3,-0.3) {\small $p_6^c$};
    \end{scope}
    \begin{scope}[shift={(8,-9)}]
      \draw[gray!40] (0,0) grid (6,6);
      \foreach \x/\y in {1/6,2/1,3/2,4/3,5/4,6/5} {\fill (\x-0.5,\y-0.5) circle (4pt);}
      \node[below] at (3,-0.3) {\small $q_6^c$};
    \end{scope}
    \begin{scope}[shift={(16,-9)}]
      \draw[gray!40] (0,0) grid (6,6);
      \foreach \x/\y in {1/5,2/6,3/4,4/3,5/2,6/1} {\fill (\x-0.5,\y-0.5) circle (4pt);}
      \node[below] at (3,-0.3) {\small $r_6^c$};
    \end{scope}
    \begin{scope}[shift={(24,-9)}]
      \draw[gray!40] (0,0) grid (6,6);
      \foreach \x/\y in {1/5,2/4,3/3,4/2,5/1,6/6} {\fill (\x-0.5,\y-0.5) circle (4pt);}
      \node[below] at (3,-0.3) {\small $s_6^c$};
    \end{scope}
  \end{tikzpicture}
  \caption{The eight patterns in $B_6$.  Top row: $p_6,q_6,r_6,s_6$.  Bottom row: their complements.}
  \label{fig:B6-patterns}
\end{figure}

With the following definition, we can give a recursive characterisation of the sets $B_k$.
\begin{definition}
  For $\pi \in \Sym_n$ and $1 \leq \ell \leq n$, the \emph{shadow} $\Sh_\ell(\pi)$ is the set of
  $\ell$-patterns contained in~$\pi$.  When $\ell=n-1$, this is exactly the set obtained by
  deleting one entry from~$\pi$ and standardising the result.
\end{definition}

\begin{proposition}\label{prop:recursive-Bk}
  For $k \geq 4$, a permutation $\pi \in \Sym_k$ belongs to $B_k$ if and only if
  $\Sh_{k-1}(\pi)$ consists of exactly one monotone pattern and one element of $B_{k-1}$.
\end{proposition}

\begin{proof}
  \textbf{Forward.}  The shadow operation commutes with reverse and complement, so
  it suffices to check the two generators $p_k, q_k$.

  For $p_k = 12\cdots(k{-}2)\,k\,(k{-}1)$: deleting any of the first $k-2$ entries leaves a
  permutation of the same shape, namely $p_{k-1}$; deleting entry $k$ or entry $k-1$ removes the
  only descent, leaving $\id_{k-1}$.  Hence
  $\Sh_{k-1}(p_k) = \{p_{k-1}, \id_{k-1}\}$.

  For $q_k = 1\,k\,(k{-}1)\cdots 2$: deleting the first entry~$1$ leaves
  $\id_{k-1}^r$; deleting any entry from the decreasing tail $k\,(k{-}1)\cdots 2$ leaves $q_{k-1}$.
  Hence $\Sh_{k-1}(q_k) = \{q_{k-1}, \id_{k-1}^r\}$.

  \medskip\noindent
  \textbf{Backward.}  Suppose $\Sh_{k-1}(\pi) = \{m, u\}$ where $m$ is monotone and
  $u \in B_{k-1}$.  By applying the complement if necessary, we may assume $m = \id_{k-1}$.
  Let $x$ be an entry whose deletion leaves $\id_{k-1}$, and let $M$ be the increasing
  $(k-1)$-pattern formed by the remaining entries.  Relative to $M$, the point $x$ lies in some cell
  $(i,j)$, meaning that exactly $i$ points of $M$ lie to the left of $x$ and exactly $j$ lie below
  $x$.

  Since deleting $x$ from $u$ leaves $\id_{k-2}$, the forward direction applied at level $k-1$
  shows that $u$ must be one of
  \[
    r_{k-1},\ q_{k-1}^c,\ s_{k-1},\ p_{k-1}.
  \]
  Equivalently, inserting one point into an increasing $(k-2)$-pattern yields
  \[
    r_{k-1}\text{ in cells }(0,1),(1,0),
  \]
  \[
    q_{k-1}^c\text{ in cell }(0,k-2),
  \]
  \[
    s_{k-1}\text{ in cell }(k-2,0),
  \]
  \[
    p_{k-1}\text{ in cells }(k-3,k-2),(k-2,k-3).
  \]

  Now delete one point of $M$ instead of deleting $x$.  Relative to the remaining increasing
  $(k-2)$-pattern, the point $x$ moves to one of
  \[
    (i,j),\ (i-1,j),\ (i,j-1),\ (i-1,j-1),
  \]
  according to whether the deleted point lies right or left of $x$, and above or below it.  Since
  every such deletion gives the same pattern $u$, every realized cell must be one of the cells listed
  above for $u$.  This forces
  \[
    u=r_{k-1}\text{ and }(i,j)\in\{(0,1),(1,0)\},
  \]
  \[
    u=q_{k-1}^c\text{ and }(i,j)=(0,k-1),
  \]
  \[
    u=s_{k-1}\text{ and }(i,j)=(k-1,0),
  \]
  \[
    u=p_{k-1}\text{ and }(i,j)\in\{(k-2,k-1),(k-1,k-2)\}.
  \]
  These are exactly the insertions producing $r_k$, $q_k^c$, $s_k$, and $p_k$, respectively, so
  $\pi \in B_k$.
\end{proof}

\begin{definition}
  Let $X \subseteq \Sym_k$.  We call $X$ \emph{monotone forcing} if there exists $N$ such that for all
  $n \geq N$,
  \[
    \Av_n(X)=\{\id_n,\id_n^r\}.
  \]
\end{definition}

A trivial example of a monotone-forcing set is $\Sym_k \setminus \{\id_k, \id_k^r\}$, since
for $n \geq k$ the only length $n$ permutations that avoid all non-monotone patterns of length $k$ are the two monotone permutations.

\section{Witness families and minimality}

For each element of $B_k$ there is an infinite family that contains only that pattern and one
monotone pattern.

\begin{definition} \label{def:Bk-witnesses}
  For $n \geq 2$, define
  \begin{alignat*}{2}
    \pi_n    & = 12\cdots(n-2)\,n\,(n-1), \qquad & \rho_n & = 1\,n\,(n-1)\cdots 2, \\
    \sigma_n & = 21\,3\,4\cdots n,               & \tau_n & = 23\cdots n\,1.
  \end{alignat*}
\end{definition}

Clearly $\pi_k = p_k$, $\rho_k = q_k$, $\sigma_k = r_k$, and $\tau_k = s_k$, but for clarity we use different notation
for the infinite families of the witness permutations compared with the patterns. The following proposition is an
analogue of Proposition~\ref{prop:recursive-Bk} for these families.

\begin{proposition}\label{prop:witness-families}
  Fix $k \geq 3$ and $n \geq k$.  Then
  \[
    \Sh_k(\pi_n) = \{12\cdots k, p_k\},\qquad
    \Sh_k(\rho_n) = \{k\cdots 21, q_k\},
  \]
  \[
    \Sh_k(\sigma_n) = \{12\cdots k, r_k\},\qquad
    \Sh_k(\tau_n) = \{12\cdots k, s_k\},
  \]
  and similarly
  \[
    \Sh_k(\pi_n^c)=\{k\cdots 21,p_k^c\},\qquad
    \Sh_k(\rho_n^c)=\{12\cdots k,q_k^c\},
  \]
  \[
    \Sh_k(\sigma_n^c)=\{k\cdots 21,r_k^c\},\qquad
    \Sh_k(\tau_n^c)=\{k\cdots 21,s_k^c\}.
  \]
\end{proposition}

\begin{proof}
  For $\pi_n$, an occurrence of a $k$-pattern either omits at least one of the final two entries, in
  which case it lies in an increasing subsequence and hence has pattern $12\cdots k$, or it uses both
  final entries, in which case the relative order is $p_k$.  The argument for $\rho_n$ is the same:
  either one omits the exceptional entry or entries and obtains a monotone
  subsequence, or one includes them and obtains $q_k$.  The claims for the
  other families follow by applying the appropriate symmetry.
\end{proof}

\begin{theorem}\label{thm:minimality}
  Every monotone-forcing subset of $\Sym_k$ contains $B_k$.
\end{theorem}

\begin{proof}
  Let $X \subseteq \Sym_k$ be monotone forcing.

  First, $X$ cannot contain either monotone pattern: if $12\cdots k \in X$, then the increasing
  permutations $\id_n$ ($n \geq k$) never belong to $\Av_n(X)$; similarly if $k\cdots 21 \in X$, then the
  decreasing permutations $\id_n^r$ ($n \geq k$) never belong to $\Av_n(X)$.

  Now suppose that some $q \in B_k$ is not in $X$.  By Proposition~\ref{prop:witness-families},
  there is an infinite family $W_n$ with $\Sh_k(W_n) = \{q, m\}$ for a single monotone pattern~$m$.
  Since neither monotone pattern belongs to $X$ and $q \notin X$, every $W_n$ avoids $X$.
  Moreover, each $W_n$ is non-monotone.  This contradicts the assumption that $X$ is
  monotone forcing.
\end{proof}

\section{Why \texorpdfstring{$B_k$}{Bk} eventually forces monotonicity}

We start with a lemma:

\begin{lemma}\label{lem:dangerous-cells}
  Let $k \geq 4$, and let $M$ be an increasing $(k-1)$-pattern.  If a point is inserted relative to
  $M$ in one of the cells
  \[
    (0,1),\ (1,0),\ (0,k-1),\ (k-1,0),\ (k-2,k-1),\ (k-1,k-2),
  \]
  then the resulting $k$-pattern lies in $B_k$.
\end{lemma}

Figure~\ref{fig:dangerous-cells-k5} shows these cells in the case $k=5$.

\begin{figure}[ht]
  \centering
  \begin{tikzpicture}[scale=0.9]
    \foreach \x/\y in {0/1,1/0,0/4,4/0,3/4,4/3} {
        \fill[gray!25] (\x,\y) rectangle ++(1,1);
        \draw[thick] (\x,\y) rectangle ++(1,1);
      }
    \draw[gray!50] (0,0) grid (5,5);
    \draw[thick] (0,0) rectangle (5,5);
    \draw[dashed, gray] (0,0) -- (5,5);
    \foreach \t in {1,2,3,4} {
        \fill (\t,\t) circle (2.8pt);
      }
    \foreach \t in {0,1,2,3,4} {
        \node[below] at (\t+0.5,-0.05) {\small $\t$};
        \node[left] at (-0.05,\t+0.5) {\small $\t$};
      }
    \node[below] at (2.5,-0.55) {\small $i$};
    \node[left] at (-0.6,2.5) {\small $j$};
    \node at (0.5,1.5) {\scriptsize $(0,1)$};
    \node at (1.5,0.5) {\scriptsize $(1,0)$};
    \node at (0.5,4.5) {\scriptsize $(0,4)$};
    \node at (4.5,0.5) {\scriptsize $(4,0)$};
    \node at (3.5,4.5) {\scriptsize $(3,4)$};
    \node at (4.5,3.5) {\scriptsize $(4,3)$};
  \end{tikzpicture}
  \caption{The cell decomposition determined by an increasing pattern $M=\id_4$ when $k=5$.
    The cell $(i,j)$ records that $i$ points of $M$ lie to the left and $j$ lie below.
    The shaded cells are the six cells from Lemma~\ref{lem:dangerous-cells}.}
  \label{fig:dangerous-cells-k5}
\end{figure}
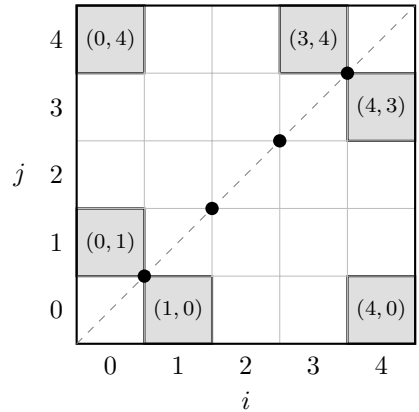

\begin{proof}
  Inserting a point into either of the cells $(0,1)$ and $(1,0)$ both produce the pattern $r_k$. Inserting into cell $(0,k{-}1)$
  produces $q_k^c$. Inserting into cell $(k{-}1,0)$ produces $s_k$, and inserting into either $(k{-}2,k{-}1)$ or $(k{-}1,k{-}2)$
  produces $p_k$.  All four distinct patterns belong to~$B_k$.
\end{proof}

The next theorem gives a general, explicit bound for when $B_k$ forces monotonicity.
Write $\mathrm{LIS}(\pi)$ and $\mathrm{LDS}(\pi)$ for the lengths of a longest increasing and a
longest decreasing subsequence of~$\pi$, respectively.

\begin{theorem}\label{thm:quadratic-bound}
  Let $k \geq 4$.  If a permutation $\pi$ avoids $B_k$ and has an increasing subsequence of length
  $3k-7$, then $\pi$ is increasing.  By symmetry, if $\pi$ avoids $B_k$ and has a decreasing
  subsequence of length $3k-7$, then $\pi$ is decreasing.

  Consequently,
  \[
    \Av_n(B_k)=\{\id_n,\id_n^r\}\qquad\text{for all }n \geq (3k-8)^2+1.
  \]
\end{theorem}

\begin{proof}
  Let $M=(m_1<\cdots<m_r)$ be a longest increasing subsequence of $\pi$, where $r \geq 3k-7$.  For a point
  $x$ not in $M$, define its cell $(i,j)$ by requiring that exactly $i$ points of $M$ lie to the left
  of $x$ and exactly $j$ points of $M$ lie below $x$.

  Now suppose that $x$ lies above the diagonal, so $i<j$.

  If $j-i \geq k-1$, then choosing $(k-1)$ points of $M$ from the gap between $i$ and $j$ produces
  a shorter increasing sequence $M'$. Relative to $M'$ the point $x$ lies in the
  cell $(0,k-1)$ and hence a pattern from $B_k$ by Lemma~\ref{lem:dangerous-cells}.

  If $r-j \geq k-2$, then choosing one point of $M$ from the gap and $(k-2)$ points above $j$
  produces the cell $(0,1)$ and again a pattern from $B_k$ by
  Lemma~\ref{lem:dangerous-cells}.

  If $i \geq k-2$, then choosing $(k-2)$ points of $M$ below $i$ and one from the gap produces the
  cell $(k-2,k-1)$ and again a pattern from $B_k$ by
  Lemma~\ref{lem:dangerous-cells}.

  Figure~\ref{fig:quadratic-proof-k5} illustrates these three choices in the case $k=5$ and
  $r=8$.

  \begin{figure}[ht]
    \centering
    \begin{tikzpicture}[scale=0.48, >=stealth]
      \begin{scope}[shift={(0,0)}, scale=0.8]
        \draw[gray!45] (0,0) grid (9,9);
        \draw[thick] (0,0) rectangle (9,9);
        \draw[dashed, gray] (0,0) -- (9,9);
        \foreach \t in {1,2,3,4,5,6,7,8} {
            \fill (\t,\t) circle (2.6pt);
          }
        \foreach \t in {2,3,4,5} {
            \fill[blue!70!black] (\t,\t) circle (3.5pt);
            \draw[->, blue!70!black, thick] (1.5,6.5) -- (\t,\t);
          }
        \fill[red!75!black] (1.5,6.5) circle (3.7pt);
        \node[red!75!black, above left] at (1.5,6.5) {\small $x$};
        \node[above] at (4.5,9.35) {\small $j-i \geq 4$};
        \node[below] at (4.5,-0.45) {\small $x$ lies in cell $(0,4)$};
        \node[below] at (4.5,-1.55) {\small relative to $M'$};
      \end{scope}

      \begin{scope}[shift={(9,0)}, scale=0.8]
        \draw[gray!45] (0,0) grid (9,9);
        \draw[thick] (0,0) rectangle (9,9);
        \draw[dashed, gray] (0,0) -- (9,9);
        \foreach \t in {1,2,3,4,5,6,7,8} {
            \fill (\t,\t) circle (2.6pt);
          }
        \foreach \x/\y in {4/4,5/5,6/6,7/7} {
            \fill[blue!70!black] (\x,\y) circle (3.5pt);
            \draw[->, blue!70!black, thick] (2.5,4.5) -- (\x,\y);
          }
        \fill[red!75!black] (2.5,4.5) circle (3.7pt);
        \node[red!75!black, above left] at (2.5,4.5) {\small $x$};
        \node[above] at (4.5,9.35) {\small $r-j \geq 3$};
        \node[below] at (4.5,-0.45) {\small $x$ lies in cell $(0,1)$};
        \node[below] at (4.5,-1.55) {\small relative to $M'$};
      \end{scope}

      \begin{scope}[shift={(18,0)}, scale=0.8]
        \draw[gray!45] (0,0) grid (9,9);
        \draw[thick] (0,0) rectangle (9,9);
        \draw[dashed, gray] (0,0) -- (9,9);
        \foreach \t in {1,2,3,4,5,6,7,8} {
            \fill (\t,\t) circle (2.6pt);
          }
        \foreach \x/\y in {2/2,3/3,4/4,5/5} {
            \fill[blue!70!black] (\x,\y) circle (3.5pt);
            \draw[->, blue!70!black, thick] (4.5,5.5) -- (\x,\y);
          }
        \fill[red!75!black] (4.5,5.5) circle (3.7pt);
        \node[red!75!black, above right] at (4.5,5.5) {\small $x$};
        \node[above] at (4.5,9.35) {\small $i \geq 3$};
        \node[below] at (4.5,-0.45) {\small $x$ lies in cell $(3,4)$};
        \node[below] at (4.5,-1.55) {\small relative to $M'$};
      \end{scope}
    \end{tikzpicture}
    \caption{The three cases in the proof of Theorem~\ref{thm:quadratic-bound}. The chosen
      $4$-point subsequence places $x$ in one of the dangerous cells from Lemma~\ref{lem:dangerous-cells}.}
    \label{fig:quadratic-proof-k5}
  \end{figure}
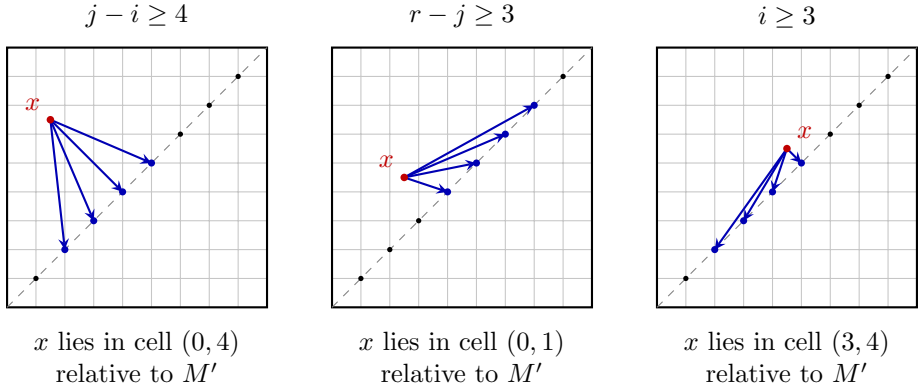

  Therefore, if $x$ does not create a pattern from $B_k$, then all three inequalities fail, and we must have
  \[
    j-i \leq k-2,\qquad r-j \leq k-3,\qquad i \leq k-3.
  \]
  Hence
  \[
    j \geq r-(k-3)\geq 2k-4,
  \]
  but also
  \[
    j \leq i+(k-2)\leq 2k-5,
  \]
  a contradiction.  So no point can lie above the diagonal.  By reversing the argument, no point can
  lie below the diagonal either.

  Thus every point of $\pi$ lies on the diagonal cells determined by $M$.  But a point on a diagonal
  cell together with $M$ implies the existence of a monotone subsequence of length $r+1$,
  contradicting the maximality of $M$.  Therefore $\pi=M$ is increasing.  The
  decreasing case is symmetric.

  Now let $\pi$ be a non-monotone permutation avoiding $B_k$.  The first part shows that
  \[
    \mathrm{LIS}(\pi)\leq 3k-8,\qquad \mathrm{LDS}(\pi)\leq 3k-8.
  \]
  By the Erd\H{o}s--Szekeres theorem, every permutation of length $(3k-8)^2+1$, or longer, has either an
  increasing or a decreasing subsequence of length $3k-7$.  Hence no non-monotone permutation of that
  length can avoid $B_k$.
\end{proof}

Let $N_k$ be the least integer such that $\Av_n(B_k)=\{\id_n,\id_n^r\}$ for all $n \geq N_k$.
Theorem~\ref{thm:quadratic-bound} shows that $N_k \leq (3k-8)^2+1$.
We defer sharper upper and lower bounds on $N_k$ to Section~\ref{sec:further-nk}.  We now use the
structure of the sets $B_k$ to prove that the reverse-reply strategy works for $k=4$ and sufficiently large $n$.

\section{Proving the Reverse-Reply Strategy for \texorpdfstring{$k=4$}{k=4}}

We now use the structural analysis of the sets $B_k$ to prove that the move-by-move reverse
strategy works for $k=4$.  By Corollary~\ref{cor:conditional-reverse}, it is enough to verify that
whenever a non-monotone pattern is legal from a reverse-closed monotone-free state, its reverse
remains legal after that move.

The proof splits into three steps.  We first handle the four reverse pairs inside $B_4$ using the
witness families from Definition~\ref{def:Bk-witnesses} in Lemma~\ref{lem:b4-cross-compatible}.
We then treat the six remaining non-monotone reverse pairs,
other than $\{2413,3142\}$, using extended witness families from Definition~\ref{def:extended-witness-family}
in Lemma~\ref{lem:reverse-reply-legality}.
Finally, we resolve the hard pair $\{2413,3142\}$ using the support sets from Definition~\ref{def:external-supports}
in Proposition~\ref{prop:hard-pair-families}.

Figure~\ref{fig:k4-proof-roadmap} summarizes this division of cases and the tools used in each branch.

\begin{figure}[ht]
  \centering
  \begin{tikzpicture}[x=0.9cm,y=0.9cm,scale=1,transform shape,>=stealth]
    \tikzstyle{box}=[draw, rounded corners, align=center, inner sep=5pt]

    \node[box, text width=4.9cm] (start) at (0,4.2)
    {reverse-closed monotone-free state $F$\\legal non-monotone move $p$};

    \node[box, text width=3.2cm] (b4) at (-5.0,1.7)
    {$p \in B_4$\\Lemma~\ref{lem:b4-cross-compatible}\\singleton witnesses};

    \node[box, text width=3.6cm] (nonhard) at (0,1.7)
    {$p \notin B_4$ and $p \notin \{2413,3142\}$\\Definition~\ref{def:extended-witness-family},\\Proposition~\ref{prop:extended-witnesses},\\Lemma~\ref{lem:reverse-reply-legality}};

    \node[box, text width=3.2cm] (hard) at (5.0,1.7)
    {$p \in \{2413,3142\}$\\Definition~\ref{def:external-supports},\\Proposition~\ref{prop:hard-pair-families}};

    \node[box, text width=5.2cm] (reply) at (0,-0.9)
    {the reverse reply $p^r$ is legal after $p$};

    \node[box, text width=4.1cm] (cor) at (0,-3.0)
    {Corollary~\ref{cor:conditional-reverse}\\reverse strategy wins};

    \node[box, text width=3.6cm] (thm) at (0,-4.8)
    {Theorem~\ref{thm:k4}\\$\Sym_n$ is a P-position for $n \geq 10$};

    \draw[->] (start) -- (b4);
    \draw[->] (start) -- (nonhard);
    \draw[->] (start) -- (hard);
    \draw[->] (b4) -- (reply);
    \draw[->] (nonhard) -- (reply);
    \draw[->] (hard) -- (reply);
    \draw[->] (reply) -- (cor);
    \draw[->] (cor) -- (thm);
  \end{tikzpicture}
  \caption{Logical flow of the $k=4$ reverse-reply proof.}
  \label{fig:k4-proof-roadmap}
\end{figure}
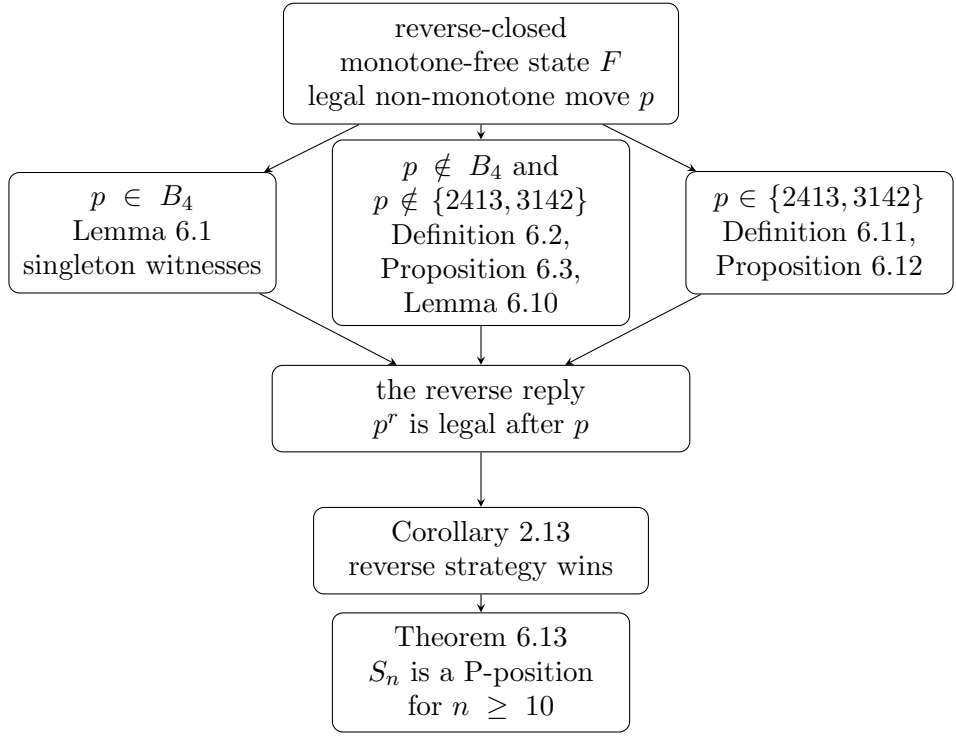

We begin with the patterns in $B_4$.

\begin{lemma}\label{lem:b4-cross-compatible}
  For each $q \in B_k$, the witness family from Proposition~\ref{prop:witness-families} for
  $q$ avoids $q^r$.  In particular, after $q$ is forbidden, $q^r$ remains legal
  (witnessed by the singleton-pattern family for $q^r$, which avoids $q$).
\end{lemma}

\begin{proof}
  Let $W_n$ be the witness permutation for $q$ from Proposition~\ref{prop:witness-families}.  Then
  $\Sh_k(W_n) = \{q, m\}$ where $m$ is a monotone pattern, and the witness $W_n'$ for $q^r$
  satisfies $\Sh_k(W_n') = \{q^r, m'\}$ where $m'$ is the other monotone pattern.  Since $q^r$ is
  non-monotone, $q^r \neq m$; since no permutation of length $\geq 2$ equals its
  reverse, $q^r \neq q$, so $q^r \notin \{q, m\}$ and $W_n$ avoids $q^r$.  Symmetrically, $q$ is
  non-monotone and hence $q \neq m'$, while $q \neq q^r$, so $q \notin \{q^r, m'\}$ and
  $W_n'$ avoids $q$.
\end{proof}

We now handle the non-$B_4$ patterns.

\begin{definition} \label{def:extended-witness-family}
  For a non-monotone $p \in \Sym_4 \setminus B_4$ and a specified companion $q \in B_4$, an
  \emph{extended witness family} for $(p,q)$ is any choice of permutations $E_n(p,q)$, one for each
  $n \geq 7$, such that
  \[
    \Sh_4(E_n(p,q)) \subseteq \{p, q, 1234, 4321\}.
  \]
\end{definition}
The cutoff $n \geq 7$ is explained by the constructions below, because it ensures
the long monotone block then has length at least $4$, so every $4$-pattern is forced either to lie
entirely in that block or to use the exceptional entries together with the monotone block.

\begin{proposition}\label{prop:extended-witnesses}
  For each of the $6$ non-$B_4$, non-monotone reverse pairs $\{p,p^r\}$ other than
  $\{2413,3142\}$, there exist two companion patterns $q_1, q_2 \in B_4$ from
  distinct reverse pairs, together with permutation families $E_n(p,q_i)$ valid for all
  $n \geq 7$, such that
  \[
    \Sh_4\bigl(E_n(p,q_i)\bigr) \setminus \{1234,4321\} = \{p, q_i\}.
  \]
  Moreover, each such family avoids $p^r$.
\end{proposition}

\begin{example}
  For
  \[
    E_n(1324,1243)=1,2,\ldots,n-3,\; n-1,\; n-2,\; n,
  \]
  the only non-monotone $4$-patterns are $1243$ and $1324$.  Reversing gives
  \[
    E_n(1324,1243){}^r=n,\; n-2,\; n-1,\; n-3,\ldots,1,
  \]
  whose only non-monotone $4$-patterns are $3421$ and $4231$.  This sample family illustrates both
  one of the witness families from Proposition~\ref{prop:extended-witnesses} and the
  reversal step used later in
  Lemma~\ref{lem:reverse-reply-legality}.  Figure~\ref{fig:extended-witness-example} shows the case
  $n=9$ together with its reverse.
\end{example}

\begin{figure}[ht]
  \centering
  \begin{tikzpicture}[scale=0.42]
    \begin{scope}[shift={(0,0)}]
      \draw[gray!40] (0,0) grid (9,9);
      \draw[thick] (0,0) rectangle (9,9);
      \foreach \x/\y in {1/1,2/2,3/3,4/4,5/5} {
          \fill[black] (\x-0.5,\y-0.5) circle (3.1pt);
        }
      \foreach \x/\y in {6/6,7/8,8/7,9/9} {
          \fill[blue!70!black] (\x-0.5,\y-0.5) circle (3.6pt);
        }
      \node[below] at (4.5,-0.45) {\small $E_9(1324,1243)=123456879$};
      \node[above] at (4.5,9.35) {\small $1324$ witness in blue};
    \end{scope}

    \begin{scope}[shift={(12,0)}]
      \draw[gray!40] (0,0) grid (9,9);
      \draw[thick] (0,0) rectangle (9,9);
      \foreach \x/\y in {1/9,2/7,3/8,4/6} {
          \fill[blue!70!black] (\x-0.5,\y-0.5) circle (3.6pt);
        }
      \foreach \x/\y in {5/5,6/4,7/3,8/2,9/1} {
          \fill[black] (\x-0.5,\y-0.5) circle (3.1pt);
        }
      \node[below] at (4.5,-0.45) {\small $E_9(1324,1243)^r=978654321$};
      \node[above] at (4.5,9.35) {\small $4231$ witness in blue};
    \end{scope}
  \end{tikzpicture}
  \caption{An instance of the family $E_n(1324,1243)$ and its reverse, shown here for
    $n=9$.}
  \label{fig:extended-witness-example}
\end{figure}
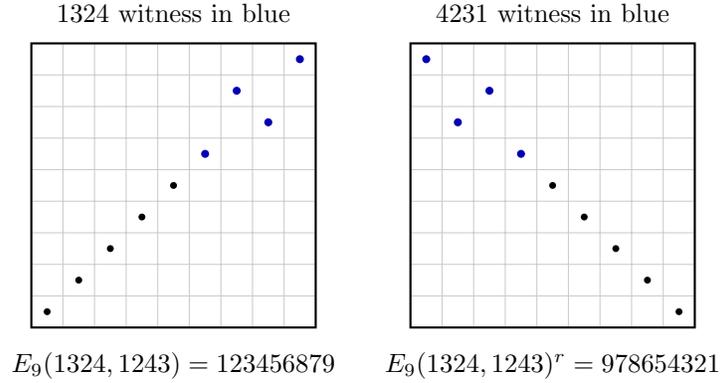

\begin{proof}
  Each displayed family in Table~\ref{tab:extended-witnesses} consists of one long monotone block together with
  at most three exceptional entries.  Consequently every non-monotone $4$-pattern must use at least
  one exceptional entry, and once the chosen exceptional entries are fixed, the tail contributes only
  monotone order.  Thus the verification reduces to a short inspection according to which exceptional
  entries are used and how many entries come from the tail.

  The preceding example checks the first row.  As a second representative case, in the row
  \[
    E_n(2143,1432)=2,\; 1,\; n,n{-}1,\ldots,3,
  \]
  choosing four entries from the decreasing tail gives the monotone pattern $4321$, choosing one of
  the first two entries together with three entries from the tail always gives $1432$, and choosing
  both initial entries together with two entries from the tail gives $2143$.  Thus the only
  non-monotone patterns are again the target and its stated companion.

  The two remaining displayed symmetry classes are checked in the same way: once the
  exceptional entries are fixed, the long tail contributes only monotone order.  The other
  $20$ ordered pairs $(p,q)$ are obtained from these four representatives by reverse, complement, and
  inverse symmetries.
\end{proof}

\begin{table}[ht]
  \centering
  \small
  \setlength{\tabcolsep}{4pt}
  \begin{tabular}{ccclc}
    \toprule
    $p$                                                   & $p^r$  & $q$    & Family $E_n(p,q)$ & $E_7(p,q)$ \\
    \midrule
    $1324$                                                & $4231$ & $1243$ &
    $\underline{1,2,\ldots,n{-}3},\; n{-}1,\; n{-}2,\; n$ &
    \begin{tabular}{@{}c@{}}
      \begin{tikzpicture}[scale=0.20,baseline=(current bounding box.center)]
        \draw[gray!35] (0,0) grid (7,7);
        \draw[thick] (0,0) rectangle (7,7);
        \foreach \x/\y in {1/1,2/2,3/3,4/4} {
            \fill[black] (\x-0.5,\y-0.5) circle (3.8pt);
          }
        \foreach \x/\y in {5/6,6/5,7/7} {
            \fill[blue!70!black] (\x-0.5,\y-0.5) circle (4.2pt);
          }
      \end{tikzpicture} \\[-1pt]
      \scriptsize $1234657$
    \end{tabular}                                                                       \\
    $1342$                                                & $2431$ & $1243$ &
    $\underline{1,2,\ldots,n{-}3},\; n{-}1,\; n,\; n{-}2$ &
    \begin{tabular}{@{}c@{}}
      \begin{tikzpicture}[scale=0.20,baseline=(current bounding box.center)]
        \draw[gray!35] (0,0) grid (7,7);
        \draw[thick] (0,0) rectangle (7,7);
        \foreach \x/\y in {1/1,2/2,3/3,4/4} {
            \fill[black] (\x-0.5,\y-0.5) circle (3.8pt);
          }
        \foreach \x/\y in {5/6,6/7,7/5} {
            \fill[blue!70!black] (\x-0.5,\y-0.5) circle (4.2pt);
          }
      \end{tikzpicture} \\[-1pt]
      \scriptsize $1234675$
    \end{tabular}                                                                       \\
    $1342$                                                & $2431$ & $2341$ &
    $1,\; \underline{3,4,\ldots,n},\; 2$                  &
    \begin{tabular}{@{}c@{}}
      \begin{tikzpicture}[scale=0.20,baseline=(current bounding box.center)]
        \draw[gray!35] (0,0) grid (7,7);
        \draw[thick] (0,0) rectangle (7,7);
        \foreach \x/\y in {2/3,3/4,4/5,5/6,6/7} {
            \fill[black] (\x-0.5,\y-0.5) circle (3.8pt);
          }
        \foreach \x/\y in {1/1,7/2} {
            \fill[blue!70!black] (\x-0.5,\y-0.5) circle (4.2pt);
          }
      \end{tikzpicture} \\[-1pt]
      \scriptsize $1345672$
    \end{tabular}                                                                       \\
    $2143$                                                & $3412$ & $1432$ &
    $2,\; 1,\; \underline{n,n{-}1,\ldots,3}$              &
    \begin{tabular}{@{}c@{}}
      \begin{tikzpicture}[scale=0.20,baseline=(current bounding box.center)]
        \draw[gray!35] (0,0) grid (7,7);
        \draw[thick] (0,0) rectangle (7,7);
        \foreach \x/\y in {3/7,4/6,5/5,6/4,7/3} {
            \fill[black] (\x-0.5,\y-0.5) circle (3.8pt);
          }
        \foreach \x/\y in {1/2,2/1} {
            \fill[blue!70!black] (\x-0.5,\y-0.5) circle (4.2pt);
          }
      \end{tikzpicture} \\[-1pt]
      \scriptsize $2176543$
    \end{tabular}                                                                       \\
    \bottomrule
  \end{tabular}
  \caption{Extended witness families for non-$B_4$ patterns (excluding $2413/3142$), listed up to
    the symmetries of the square.  The four displayed rows represent all $24$ ordered pairs $(p,q)$:
    every other pair is obtained from one of these by applying reverse, complement, or inverse to the
    pair $(p,q)$ and to the witness family.  The underlined segment is the long monotone block in each
    family.  In each plot of $E_7(p,q)$, the blue points are the exceptional entries and the black
    points form the long monotone block.}
  \label{tab:extended-witnesses}
\end{table}

\begin{example}
  Let
  \[
    F=\{1432,2341\}.
  \]
  Then $F$ is reverse-closed, contains no monotone pattern, and omits the reverse pair
  $1243/3421$ from $B_4$. The family $E_n(1324,1243)$ above shows that $1324$ is legal from $\Av_n(F)$,
  so if Player I chooses $1324$ we can use the family $E_n(4231,3421) = E_n(1324,1243)^r$ (which lies in $\Av_n(F\cup\{1324\})$)
  to see that the reverse reply $4231$ is legal.  This is the first branch of the non-$B_4$ argument used later in
  Lemma~\ref{lem:reverse-reply-legality}.
\end{example}

To treat the remaining states in which both companion reverse pairs are already forbidden, we now
introduce one-entry inflation families.

\begin{definition}
  Fix $\beta \in \Sym_6$, an entry index $i \in \{1,\dots,6\}$, and a sign $\varepsilon \in \{+,-\}$.
  For $t \geq 1$, let $I_t(\beta,i,\varepsilon)$ be the permutation obtained by inflating the $i$th
  entry of $\beta$ into an increasing block of length $t$ if $\varepsilon=+$, and into a decreasing
  block of length $t$ if $\varepsilon=-$.

  Define the \emph{stable non-monotone shadow}
  \[
    \Sigma(\beta,i,\varepsilon)=\bigcup_{t=1}^4
    \bigl(\Sh_4(I_t(\beta,i,\varepsilon))\setminus\{1234,4321\}\bigr).
  \]
\end{definition}

\begin{example}
  Take $\beta=123465$, $i=5$, and $\varepsilon=-$.  Then
  \[
    I_1(\beta,5,-)=123465
  \]
  has non-monotone shadow
  \[
    \Sh_4(I_1(\beta,5,-))\setminus\{1234,4321\}=\{1243\},
  \]
  whereas the stable non-monotone shadow is
  \[
    \Sigma(\beta,5,-)=\{1243,1432\}.
  \]
  The extra pattern $1432$ already appears at block size $2$, since
  \[
    I_2(\beta,5,-)=1234765
  \]
  and
  \[
    \Sh_4(I_2(\beta,5,-))\setminus\{1234,4321\}=\{1243,1432\}.
  \]
  So in this family the ordinary non-monotone shadow depends on the block size, while
  $\Sigma(\beta,5,-)$ records the eventual value.
\end{example}

For later use with the hard pair, Figure~\ref{fig:stable-shadow-examples} shows the two
block-size-$4$ families $I_4(241356,5,+)$ and $I_4(314256,5,+)$. Their stable non-monotone shadows
are
\[
  \Sigma(241356,5,+)=\{1324,2134,2314,2413,3124\},
\]
\[
  \Sigma(314256,5,+)=\{1324,2134,2314,3124,3142\}.
\]
Thus the two stable shadows differ only in whether they contain $2413$ or $3142$.  Removing that
hard pair and then recording the remaining non-monotone patterns only up to reverse-pairs gives the
same set in both cases:
\[
  \{1324/4231,\; 2134/4312,\; 2314/4132,\; 3124/4213\}.
\]
This viewpoint becomes important later in Definition~\ref{def:external-supports} and
Proposition~\ref{prop:hard-pair-families}, because for a reverse-closed forbidden set the relevant
question is which non-hard reverse pairs remain outside the distinguished hard pair
$\{2413,3142\}$.

The next lemma explains why once the inflated block has length at least $4$, the non-monotone
shadow has stabilized to $\Sigma(\beta,i,\varepsilon)$.

\begin{figure}[ht]
  \centering
  \begin{tikzpicture}[scale=0.42]
    \begin{scope}[shift={(0,0)}]
      \fill[blue!8] (4,4) rectangle (8,8);
      \draw[gray!40] (0,0) grid (9,9);
      \draw[thick] (0,0) rectangle (9,9);
      \foreach \x/\y in {1/2,2/4,3/1,4/3} {
          \fill[blue!70!black] (\x-0.5,\y-0.5) circle (3.6pt);
        }
      \foreach \x/\y in {5/5,6/6,7/7,8/8,9/9} {
          \fill[black] (\x-0.5,\y-0.5) circle (3.1pt);
        }
      \node[below] at (4.5,-0.45) {\small $241356789$};
      \node[above] at (4.5,9.35) {\small $I_4(241356,5,+)$};
      \node[blue!70!black] at (2.35,4.95) {\small $2413$};
    \end{scope}

    \begin{scope}[shift={(12,0)}]
      \fill[blue!8] (4,4) rectangle (8,8);
      \draw[gray!40] (0,0) grid (9,9);
      \draw[thick] (0,0) rectangle (9,9);
      \foreach \x/\y in {1/3,2/1,3/4,4/2} {
          \fill[blue!70!black] (\x-0.5,\y-0.5) circle (3.6pt);
        }
      \foreach \x/\y in {5/5,6/6,7/7,8/8,9/9} {
          \fill[black] (\x-0.5,\y-0.5) circle (3.1pt);
        }
      \node[below] at (4.5,-0.45) {\small $314256789$};
      \node[above] at (4.5,9.35) {\small $I_4(314256,5,+)$};
      \node[blue!70!black] at (2.35,4.95) {\small $3142$};
    \end{scope}
  \end{tikzpicture}
  \caption{Two block-size-$4$ one-entry inflation examples.  The shaded square
    marks the inflated block, and the blue points show the distinguished hard pattern: $2413$ on the
    left and $3142$ on the right.}
  \label{fig:stable-shadow-examples}
\end{figure}
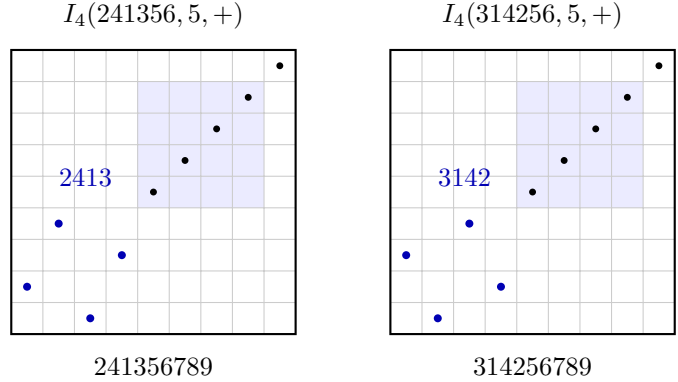

\begin{lemma}\label{lem:stable-oneblock-shadow}
  For every triple $(\beta,i,\varepsilon)$ and every $t \geq 4$,
  \[
    \Sh_4(I_t(\beta,i,\varepsilon))\setminus\{1234,4321\}=\Sigma(\beta,i,\varepsilon).
  \]
\end{lemma}

\begin{proof}
  Any $4$-subsequence of $I_t(\beta,i,\varepsilon)$ uses exactly $j$ entries from the inflated block
  for some $0 \leq j \leq 4$.  Once $t \geq j$, the standardized pattern of such a subsequence depends
  only on $j$, on the sign $\varepsilon$, and on the choice of the other $4-j$ base entries.  Hence
  every non-monotone $4$-pattern that occurs for some $t \geq 4$ already occurs for block size
  $\max(1,j) \leq 4$, so it lies in $\Sigma(\beta,i,\varepsilon)$.  The converse inclusion is
  immediate from the definition of $\Sigma(\beta,i,\varepsilon)$.
\end{proof}

\begin{proposition}\label{prop:residual-nonhard-inflations}
  Let $F \subseteq \Sym_4$ be reverse-closed and contain neither monotone pattern, and let
  $p \notin B_4 \cup \{2413,3142\}$ be a non-monotone pattern.  Suppose that $p$ is legal from
  $\Av_7(F)$ and that both companion reverse pairs for $p$ from
  Proposition~\ref{prop:extended-witnesses} lie entirely in~$F$.  Then there exists a triple
  $(\beta,i,\varepsilon)$ such that
  \[
    p^r \in \Sigma(\beta,i,\varepsilon),\qquad p \notin \Sigma(\beta,i,\varepsilon),
  \]
  and every other non-monotone pattern in $\Sigma(\beta,i,\varepsilon)$ belongs to a reverse pair
  disjoint from~$F$.
\end{proposition}

\begin{proof}
  Since $F$ is reverse-closed, we may identify it with the set of non-monotone reverse pairs that it
  contains.  The statement is then a finite  computation.  The script
  \path{code-permuta/k4_reverse_reply_permuta.py} in the public repository~\cite{game-patterns-repo}
  enumerates all pairs $(F,p)$ satisfying the hypotheses; there are
  $1156$ ($297$ if we reduce them by symmetries) such pairs.  For each of them, it finds a triple $(\beta,i,\varepsilon)$,
  with $\beta \in \Sym_6$, whose stable
  non-monotone shadow has the required properties.  See Section~\ref{sec:repro} for details on reproducibility.
\end{proof}

\begin{lemma}\label{lem:reverse-reply-legality}
  For $k=4$ and $n \geq 9$, whenever $F \subseteq \Sym_4$ is reverse-closed, contains neither
  monotone pattern, and $p \notin F$ is a non-monotone legal pattern with
  $p \notin \{2413,3142\}$, the reverse pattern $p^r$ is legal from $\Av_n(F \cup \{p\})$.
\end{lemma}

\begin{proof}
  We split into two cases.

  \medskip\noindent
  \textbf{Case~1: $p \in B_4$.}
  By Lemma~\ref{lem:b4-cross-compatible}, the witness family for $p^r$ has
  $\Sh_4 = \{p^r, m\}$ where $m$ is monotone.  Since $p \neq p^r$ and $p \neq m$,
  this family avoids $p$.  Furthermore, since $p^r \notin F$ (by reverse-closure of $F$) and
  $m \notin F$, the family lies in
  $\Av_n(F \cup \{p\})$ and witnesses the legality of $p^r$.

  \medskip\noindent
  \textbf{Case~2: $p \notin B_4$, $p \neq 2413,3142$.}
  By Proposition~\ref{prop:extended-witnesses}, each such $p$ has two companions
  $q_1, q_2 \in B_4$ from distinct reverse pairs
  (listed up to symmetry in Table~\ref{tab:extended-witnesses}).  If at least one companion $q_i$
  satisfies $q_i \notin F$,
  then the reversed family $E_n(p,q_i){}^r$ has
  $\Sh_4 \setminus \{1234,4321\} = \{p^r, q_i^r\}$ and avoids $p$.
  Since $p^r \notin F$ (reverse-closure), $q_i^r \notin F$ (because $q_i \notin F$
  and reverse-closure), and $F$ contains no monotone pattern, this family lies in
  $\Av_n(F \cup \{p\})$ and witnesses the legality of $p^r$.

  Otherwise both companion reverse pairs lie entirely in $F$.  Choose
  $\pi \in \Av_n(F)$ containing $p$.  Restricting to the entries of one occurrence of $p$ together
  with any three additional entries of $\pi$ shows that $p$ is legal from $\Av_7(F)$.  Hence
  Proposition~\ref{prop:residual-nonhard-inflations} gives a triple
  $(\beta,i,\varepsilon)$ such that $p^r \in \Sigma(\beta,i,\varepsilon)$ and
  $p \notin \Sigma(\beta,i,\varepsilon)$.  Moreover, every other non-monotone pattern in
  $\Sigma(\beta,i,\varepsilon)$ belongs to a reverse pair disjoint from $F$.
  Taking block size $t=n-5 \geq 4$, Lemma~\ref{lem:stable-oneblock-shadow} gives a length-$n$
  permutation $I_t(\beta,i,\varepsilon)$ whose non-monotone $4$-shadow is exactly
  $\Sigma(\beta,i,\varepsilon)$.  This permutation therefore avoids $F \cup \{p\}$ and contains
  $p^r$, so $p^r$ is legal from $\Av_n(F \cup \{p\})$.
\end{proof}

Finally, we handle the hard pair $\{2413,3142\}$.

\begin{definition}\label{def:external-supports}
  For $p \in \{2413,3142\}$, let
  \[
    \mathcal{E}_p=
    \left\{
    \begin{array}{l}
      \Sigma(\beta,i,\varepsilon)\setminus\{2413,3142\}:           \\
      \beta\in \Sym_6,\ i\in\{1,\dots,6\},\ \varepsilon\in\{+,-\}, \\
      p \in \Sigma(\beta,i,\varepsilon),\ p^r \notin \Sigma(\beta,i,\varepsilon)
    \end{array}
    \right\},
  \]
  where each set is recorded only up to reverse-pairs.  In other words, $\mathcal{E}_p$ is the set
  of reverse-pair sets obtained by taking all one-entry inflation families
  $I_t(\beta,i,\varepsilon)$ with $\beta\in\Sym_6$, keeping those whose stable shadow contains $p$
  but not $p^r$, and then deleting the hard pair $\{2413,3142\}$.  We call $\mathcal{E}_p$ the
  \emph{collection of external reverse-pair supports} for~$p$.
\end{definition}

\begin{proposition}\label{prop:hard-pair-families}
  The two collections $\mathcal{E}_{2413}$ and $\mathcal{E}_{3142}$ are equal, and each has size
  $138$.  Moreover, if $F \subseteq \Sym_4$ is reverse-closed, contains neither monotone pattern, and
  $2413$ is legal from $\Av_7(F)$, then the reverse-pair set corresponding to $F$ is disjoint from
  some support in $\mathcal{E}_{2413}$.  By reversal, the analogous statement also holds with $2413$
  and $3142$ interchanged.
\end{proposition}

\begin{proof}
  Reversal sends $2413$ to $3142$, and for every one-entry inflation family
  $I_t(\beta,i,\varepsilon)$ its reverse is again a one-entry inflation family:
  \[
    I_t(\beta,i,+)^r = I_t(\beta^r,7-i,-),
    \qquad
    I_t(\beta,i,-)^r = I_t(\beta^r,7-i,+),
  \]
  so the definition of $\mathcal{E}_p$ is closed under reversal.  Moreover,
  taking shadows commutes with reversal, and the external supports are recorded only up to
  reverse-pairs.  Therefore reversal gives a bijection from the families in
  $\mathcal{E}_{2413}$ to those in $\mathcal{E}_{3142}$, so
  \[
    \mathcal{E}_{2413}=\mathcal{E}_{3142}.
  \]

  Since $F$ is reverse-closed, we identify it with the set of non-monotone reverse pairs that it
  contains.  The remaining assertions are finite exact computations.  The script
  \path{code-permuta/k4_reverse_reply_permuta.py} in the public repository~\cite{game-patterns-repo}
  enumerates all $6!\cdot 6 \cdot 2$ one-entry inflation families,
  computes their stable shadows $\Sigma(\beta,i,\varepsilon)$, and records the distinct external
  reverse-pair supports for those containing $2413$ but not $3142$.  It finds $138$ such supports,
  and hence also $138$ supports in $\mathcal{E}_{3142}$ by the equality above.

  The same script then enumerates all reverse-closed monotone-free sets $F \subseteq \Sym_4$ with
  $2413,3142 \notin F$ for which $2413$ is legal at $n=7$; there are $543$ such states.  For each of
  them, at least one member of $\mathcal{E}_{2413}$ is disjoint from $F$.  The final symmetry
  statement follows by reversing every surviving permutation.  See Section~\ref{sec:repro}.
\end{proof}

\begin{theorem}\label{thm:k4}
  For every $n \geq 10$, the move-by-move reverse strategy is a winning strategy for Player~II in PAP
  on $\Sym_n$ with $k=4$.  This implies,
  \[
    \sg(\Sym_n,4)=0 \qquad \text{for } n \geq 10.
  \]
\end{theorem}

\begin{example}
  Suppose $F \subseteq \Sym_4$ is reverse-closed, contains neither monotone pattern, and omits the
  reverse pairs
  \[
    \{1324/4231,\; 2134/4312,\; 2314/4132,\; 3124/4213\}
  \]
  and let $n=10$.  Then
  \[
    I_5(314256,5,+)=(3,1,4,2,5,6,7,8,9,10)
  \]
  lies in $\Av_{10}(F\cup\{2413\})$ and contains $3142$.  The proof of
  Theorem~\ref{thm:k4} shows that every legal hard-pair state at $n \geq 10$ admits some witness of
  this form, after first passing to a $7$-point subpermutation and then choosing an appropriate stable
  support from Proposition~\ref{prop:hard-pair-families}.
\end{example}

\begin{proof}
  By Lemma~\ref{lem:reverse-reply-legality}, only the pair $\{2413,3142\}$ remains.  Let
  $F \subseteq \Sym_4$ be reverse-closed, contain neither monotone pattern, and suppose that $2413$ is
  legal from $\Av_n(F)$, where $n \geq 10$.

  Choose $\pi \in \Av_n(F)$ containing $2413$, and select a $4$-point occurrence of $2413$ in $\pi$.
  Together with any three additional entries of $\pi$, these points form a $7$-point subpermutation
  still avoiding $F$ and still containing $2413$.  Thus $2413$ is legal from $\Av_7(F)$.

  By Proposition~\ref{prop:hard-pair-families}, the reverse-pair set of $F$ is disjoint from some
  support $E \in \mathcal{E}_{2413}$.  Since $\mathcal{E}_{2413}=\mathcal{E}_{3142}$, there is a
  triple $(\beta,i,\varepsilon)$ whose stable shadow contains $3142$, avoids $2413$, and has external
  support $E$.  Taking block size $t=n-5 \geq 5$, Lemma~\ref{lem:stable-oneblock-shadow} gives a
  length-$n$ permutation
  $I_t(\beta,i,\varepsilon)$ whose non-monotone $4$-shadow consists of $3142$ together with patterns
  belonging to the reverse pairs in $E$, and hence avoids $F \cup \{2413\}$.  So $3142$ is legal after
  $2413$ is chosen.  The case with $2413$ and $3142$ interchanged is symmetric.

  Corollary~\ref{cor:conditional-reverse} now applies.
\end{proof}

\subsection{Toward general \texorpdfstring{$k$}{k}}

The $k=4$ proof suggests a plausible approach to $k=5$.
First one reduces, via Corollary~\ref{cor:conditional-reverse}, to reverse-closed
monotone-free states.  Next the four reverse pairs inside $B_k$ are handled uniformly by the
singleton-shadow witness families.  One would then try to treat the remaining patterns by bounded witness
families whose non-monotone shadows involve only a small amount of extra support, and finally one
isolates and analyzes the residual hard reverse pairs that survive those first passes.

The combinatorics of the last step, however, should become substantially more complicated as $k$
grows.  There are
\[
  \frac{k!-2}{2}
\]
non-monotone reverse pairs in total, and after removing the four reverse pairs coming from $B_k$
there remain
\[
  \frac{k!-10}{2}
\]
non-$B_k$ reverse pairs that could potentially be hard.  For $k=4$ this number is $7$, and the
explicit companion families reduce it all the way down to the single pair $\{2413,3142\}$.  For
$k=5$ the corresponding number is already $55$.  The present computations show that the simplest
one-companion construction works for $14$ of these, but this still leaves
many unresolved cases, so one should not expect a single hard pair to remain.

This suggests that an extension of the $k=4$ proof to $k=5$ should still be possible, but
it is likely to be more nuanced and case-based, with many more exceptional configurations to treat.
For general $k$, however, one should probably not expect a direct iteration of the $k=4$ argument to
be enough: a complete proof of Conjecture~\ref{conj:general-reverse} will likely require a new idea
beyond the present witness-family analysis.

One potential general approach is trying to use the fact that the hard pair for $k=4$
consists of exactly the simple permutations of length $4$, so perhaps one can use the
theory of simple permutations to push to general $k$.

\section{Further bounds on \texorpdfstring{$N_k$}{Nk}}\label{sec:further-nk}

We now return to the threshold $N_k$.  Theorem~\ref{thm:quadratic-bound} gave the bound
$N_k \leq (3k-8)^2+1$.  We can improve this upper bound using more of the patterns in $B_k$.

\begin{proposition}\label{prop:intermediate-lis-lds}
  Let $k \geq 4$, and let $\pi \in \Av(B_k)$.  If $\mathrm{LIS}(\pi) \geq 2k-4$, then
  \[
    \mathrm{LDS}(\pi)\leq k-2.
  \]
  By symmetry, if $\mathrm{LDS}(\pi) \geq 2k-4$, then $\mathrm{LIS}(\pi)\leq k-2$.
\end{proposition}

\begin{proof}
  Let $M=(m_1<\cdots<m_r)$ be a longest increasing subsequence of $\pi$, where $r \geq 2k-4$.  For a
  point $x$ not in $M$, define its cell $(i,j)$ as in the proof of
  Theorem~\ref{thm:quadratic-bound}.

  We claim that every such point satisfies $i \geq 1$ and $j \geq 1$.

  If $i=0$ and $j=0$, then $x$ together with $M$ forms an increasing subsequence of length $r+1$,
  contradicting the maximality of $M$.

  If $i=0$ and $1 \leq j \leq k-2$, then $r-j \geq k-2$, so choosing one point of $M$ below $x$ and
  $k-2$ points above $x$ produces the cell $(0,1)$ and hence a pattern from $B_k$ by
  Lemma~\ref{lem:dangerous-cells}.

  If $i=0$ and $j \geq k-1$, then choosing $k-1$ points of $M$ below $x$ produces the cell $(0,k-1)$
  and again a pattern from $B_k$ by Lemma~\ref{lem:dangerous-cells}.

  Therefore $i \neq 0$.  By symmetry, the cases $(1,0)$ and $(k-1,0)$ from
  Lemma~\ref{lem:dangerous-cells} show that $j \neq 0$ as well.

  So $m_1$ lies to the left of and below every other point of $\pi$.  Now suppose that $\pi$ has a
  decreasing subsequence of length $k-1$.  Since $m_1$ is the least point of $\pi$, it cannot belong
  to a decreasing subsequence of length greater than $1$.  Prepending $m_1$ to such a subsequence
  therefore produces the pattern
  \[
    q_k=1\,k\,(k-1)\cdots 2,
  \]
  which belongs to $B_k$, a contradiction.
\end{proof}

\begin{corollary}\label{cor:better-quadratic-bound}
  Let $k \geq 4$.  Then
  \[
    \Av_n(B_k)=\{\id_n,\id_n^r\}\qquad\text{for all }n \geq (2k-5)^2+1.
  \]
  Equivalently,
  \[
    N_k \leq (2k-5)^2+1.
  \]
\end{corollary}

\begin{proof}
  Let $\pi$ be a non-monotone permutation avoiding $B_k$.  By Theorem~\ref{thm:quadratic-bound},
  \[
    \mathrm{LIS}(\pi)\leq 3k-8,\qquad \mathrm{LDS}(\pi)\leq 3k-8.
  \]

  If $|\pi|>(2k-5)^2$, then the Erd\H{o}s--Szekeres theorem gives either
  \[
    \mathrm{LIS}(\pi)\geq 2k-4
  \]
  or
  \[
    \mathrm{LDS}(\pi)\geq 2k-4.
  \]

  In the first case, Proposition~\ref{prop:intermediate-lis-lds} gives
  \[
    \mathrm{LDS}(\pi)\leq k-2,
  \]
  so
  \[
    |\pi|\leq (3k-8)(k-2).
  \]

  In the second case, the symmetric statement in Proposition~\ref{prop:intermediate-lis-lds} gives
  \[
    \mathrm{LIS}(\pi)\leq k-2,
  \]
  and again
  \[
    |\pi|\leq (3k-8)(k-2).
  \]

  But for $k \geq 4$ we have
  \[
    (3k-8)(k-2)<(2k-5)^2,
  \]
  since the difference is $(k-3)^2$.  This contradiction shows that no non-monotone permutation of
  length $(2k-5)^2+1$ can avoid $B_k$.
\end{proof}

The bound above is still far from the largest avoiders found computationally.
The next two constructions give quadratic lower bounds.
\begin{definition}
  For $k \geq 4$, let $E_k$ be the \emph{staircase permutation} obtained by concatenating an increasing block $A$ of
  length $k-3$, then $k-2$ increasing blocks $C_1,\dots,C_{k-2}$ of length $k-2$ arranged so that
  every entry of $C_i$ is larger than every entry of $C_j$ whenever $i<j$, and finally another
  increasing block $B$ of length $k-3$.  Thus
  \[
    |E_k|=(k-3)+(k-2)^2+(k-3)=k^2-2k-2.
  \]
\end{definition}

\begin{proposition}\label{prop:staircase-lower-bound}
  For every $k \geq 4$, the staircase permutation $E_k$ avoids $B_k$.  Consequently,
  \[
    N_k \geq k^2-2k-1.
  \]
\end{proposition}

\begin{proof}
  For this proof, write
  \[
    G(a;c_1,\dots,c_m;b)=\id_a \oplus (\id_{c_1}\ominus \cdots \ominus \id_{c_m}) \oplus \id_b.
  \]
  Then
  \[
    E_k=G(k-3;\underbrace{k-2,\dots,k-2}_{k-2};k-3).
  \]
  Every pattern contained in $E_k$ is again of the form $G(a;c_1,\dots,c_m;b)$, with parameters
  weakly smaller than those of $E_k$.  In particular, any pattern in $E_k$ satisfies
  \[
    a\leq k-3,\qquad b\leq k-3,\qquad m\leq k-2,\qquad c_i\leq k-2\ \text{ for all }i.
  \]

  The eight patterns in $B_k$ appear in this family as
  \begin{align*}
    p_k   & = G(k-2;1,1;0),                          &
    q_k   & = G(1;\underbrace{1,\dots,1}_{k-1};0),     \\
    r_k   & = G(0;1,1;k-2),                          &
    s_k   & = G(0;k-1,1;0),                            \\
    p_k^c & = G(0;\underbrace{1,\dots,1}_{k-2},2;0), &
    q_k^c & = G(0;1,k-1;0),                            \\
    r_k^c & = G(0;2,\underbrace{1,\dots,1}_{k-2};0), &
    s_k^c & = G(0;\underbrace{1,\dots,1}_{k-1};1).
  \end{align*}

  Now $p_k$ and $r_k$ are impossible because they require $a=k-2$ or $b=k-2$, whereas every pattern
  in $E_k$ has $a,b\leq k-3$.  The patterns $q_k$, $p_k^c$, $r_k^c$, and $s_k^c$ are impossible
  because they require at least $k-1$ skew blocks, whereas every pattern in $E_k$ has at most $k-2$
  such blocks.  Finally, $s_k$ and $q_k^c$ are impossible because they require a skew block of size
  $k-1$, whereas every block occurring in a pattern of $E_k$ has size at most $k-2$.

  Therefore $E_k$ avoids all eight patterns in $B_k$.  Since $|E_k|=k^2-2k-2$, the threshold
  satisfies $N_k \geq k^2-2k-1$.
\end{proof}

\begin{proposition}\label{prop:double-layered-lower-bound}
  For $k \geq 4$, let
  \[
    L_k=\underbrace{\id_{k-3}\ominus \cdots \ominus \id_{k-3}}_{k-2\text{ summands}}
  \]
  and set
  \[
    H_k=L_k\oplus L_k.
  \]
  Then $H_k$ avoids $B_k$.  Consequently,
  \[
    N_k \geq 2(k-2)(k-3)+1.
  \]
\end{proposition}

\begin{proof}
  For this proof, write
  \[
    D(a_1,\dots,a_m\,;\,b_1,\dots,b_n)
    =(\id_{a_1}\ominus\cdots\ominus\id_{a_m})\oplus(\id_{b_1}\ominus\cdots\ominus\id_{b_n}),
  \]
  where either list may be empty.  Then
  \[
    H_k=D(\underbrace{k-3,\dots,k-3}_{k-2}\,;\,\underbrace{k-3,\dots,k-3}_{k-2}).
  \]
  Every pattern contained in $H_k$ is again of the form $D(a_1,\dots,a_m\,;\,b_1,\dots,b_n)$, with
  weakly smaller parameters.  In particular, any pattern in $H_k$ satisfies
  \[
    m\leq k-2,\qquad n\leq k-2,\qquad a_i\leq k-3,\qquad b_j\leq k-3.
  \]

  The eight patterns in $B_k$ appear in this family as
  \begin{align*}
    p_k   & = D(k-2\,;\,1,1),                         &
    q_k   & = D(1\,;\,\underbrace{1,\dots,1}_{k-1}),    \\
    r_k   & = D(1,1\,;\,k-2),                         &
    s_k   & = D(\underbrace{1,\dots,1}_{k-1}\,;\,1),    \\
    p_k^c & = D(\underbrace{1,\dots,1}_{k-2},2\,;\,), &
    q_k^c & = D(1,k-1\,;\,),                            \\
    r_k^c & = D(\,;\,2,\underbrace{1,\dots,1}_{k-2}), &
    s_k^c & = D(\,;\,\underbrace{1,\dots,1}_{k-2},2).
  \end{align*}

  Now $p_k$ and $r_k$ are impossible because they require a block of size $k-2$, whereas every block
  in a pattern of $H_k$ has size at most $k-3$.  The patterns $q_k$, $s_k$, $p_k^c$, and $s_k^c$
  are impossible because they require $k-1$ skew blocks on one side, whereas each side of a pattern
  in $H_k$ has at most $k-2$ such blocks.  Finally, $q_k^c$ and $r_k^c$ are impossible because they
  require a block of size $k-1$, again larger than any block occurring in a pattern of $H_k$.

  Thus $H_k$ avoids all eight patterns in $B_k$.  Since
  \[
    |H_k|=2(k-2)(k-3),
  \]
  the threshold satisfies $N_k \geq 2(k-2)(k-3)+1$.
\end{proof}

\begin{example}\label{ex:H6}
  The first case where the family $H_k$ improves the staircase lower bound is $k=6$.  Here
  \[
    L_6=10\,11\,12\,7\,8\,9\,4\,5\,6\,1\,2\,3,
  \]
  so
  \[
    H_6=L_6\oplus L_6
    =10\,11\,12\,7\,8\,9\,4\,5\,6\,1\,2\,3\,22\,23\,24\,19\,20\,21\,16\,17\,18\,13\,14\,15.
  \]
  Thus $H_6$ is the direct sum of two copies of a $4$-block layered permutation, each block being an
  increasing sequence of length $3$.  Figure~\ref{fig:H6} shows this explicitly.
\end{example}

\begin{figure}[ht]
  \centering
  \begin{tikzpicture}[scale=0.27]
    \fill[gray!10] (0,0) rectangle (12,12);
    \fill[gray!10] (12,12) rectangle (24,24);

    \draw[gray!55] (0,0) grid (24,24);
    \draw[thick] (0,0) rectangle (24,24);
    \draw[very thick] (0,0) rectangle (12,12);
    \draw[very thick] (12,12) rectangle (24,24);
    \draw[dashed] (12,0) -- (12,24);
    \draw[dashed] (0,12) -- (24,12);

    \foreach \x/\y in {0/9,3/6,6/3,9/0,12/21,15/18,18/15,21/12} {
        \draw[thick] (\x,\y) rectangle ++(3,3);
      }

    \foreach \x/\y in {
        1/10,2/11,3/12,4/7,5/8,6/9,7/4,8/5,9/6,10/1,11/2,12/3
      } {
        \fill[black] (\x-0.5,\y-0.5) circle (3.2pt);
      }
    \foreach \x/\y in {
        13/22,14/23,15/24,16/19,17/20,18/21,19/16,20/17,21/18,22/13,23/14,24/15
      } {
        \fill[black] (\x-0.5,\y-0.5) circle (3.2pt);
      }

    \node at (6,12.9) {\small $L^-$};
    \node at (18,24.9) {\small $L^+$};
  \end{tikzpicture}
  \caption{The permutation $H_6=L_6\oplus L_6$.  Each $3\times 3$ box is an increasing block;
    the two larger highlighted squares are the copies $L^-$ and $L^+$.}
  \label{fig:H6}
\end{figure}
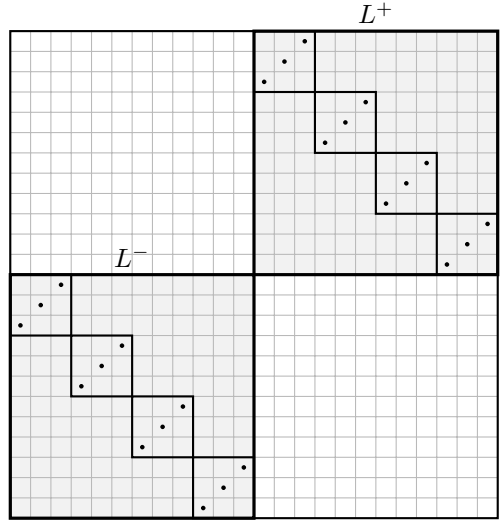

We now have
\[
  \max\{k^2-2k-1,\ 2(k-2)(k-3)+1\}\leq N_k\leq (2k-5)^2+1\qquad\text{for all }k\geq 4,
\]
while $N_3=1$.  In particular,
\[
  7 \leq N_4 \leq 10,\qquad 14 \leq N_5 \leq 26,\qquad 25 \leq N_6 \leq 50.
\]

\begin{proposition}\label{prop:exact-thresholds}
  We have
  \[
    N_3=1,\qquad N_4=7,\qquad N_5=14,\qquad N_6=25.
  \]
\end{proposition}

\begin{proof}
  For $k=3$, we have $B_3=\{132,213,231,312\}$, so every permutation of length at least $1$ avoiding
  $B_3$ is monotone, and $N_3=1$.

  For $k=4,5,6$, the lower bounds from Proposition~\ref{prop:staircase-lower-bound} and
  Proposition~\ref{prop:double-layered-lower-bound} already give
  \[
    N_4\geq 7,\qquad N_5\geq 14,\qquad N_6\geq 25.
  \]
  The classes $\Av_n(B_k)$ were then enumerated until the threshold value was reached.
  For $k=4,5,6$, this happens at $n=7,14,25$, respectively.

  Once this happens at some length $m \geq k$, it persists for all larger lengths.
  These computations are reproduced in
  Section~\ref{sec:repro}; except for $k=6$ which used the incremental fast avoider-generation methods of
  Kuszmaul~\cite{kuszmaul}.
\end{proof}

\section{Concluding remarks}

The sets $B_k$ seem to be the right finite obstruction sets for single-length PAP\@.  They are
minimal by Theorem~\ref{thm:minimality}, eventually monotone forcing by
Theorem~\ref{thm:quadratic-bound} and Corollary~\ref{cor:better-quadratic-bound}, and the thresholds for
$k=3,4,5,6$, are
$N_3=1$, $N_4=7$, $N_5=14$, and $N_6=25$.  For $k=3$, the reverse strategy is completely understood
(Theorem~\ref{thm:k3}); for $k=4$, the reverse strategy is now settled as well, with
exceptions at $n=5,6,7,8,9$ and a large-$n$ theorem proved via stable one-entry inflation
families.  The staircase permutations $E_k$ from
Proposition~\ref{prop:staircase-lower-bound} show that
\[
  N_k\geq k^2-2k-1,
\]
while Proposition~\ref{prop:double-layered-lower-bound} gives the stronger bound
\[
  N_k\geq 2(k-2)(k-3)+1
\]
for $k \geq 6$, and already yields $N_6 \geq 25$.

The next step is to push the same program to $k=5$ and beyond
(Conjecture~\ref{conj:general-reverse}).

There are also several natural game-theoretic variants that are not pursued here.  Since PAP
positions are finite impartial games with well-defined Sprague--Grundy values, one can consider
disjunctive sums of PAP positions, for example sums of starting positions $\Sym_n$ with different
pattern lengths.  The normal-play convention is also not essential: because every play has length at
most $k!$, a mis\`ere version is just as natural, although the reverse-reply arguments used here do
not immediately transfer to that setting.  Finally, one could study partizan versions in which the
two players are allowed to choose from different sets of patterns, or even from different allowed
pattern lengths.  These variants seem likely to require ideas beyond the present paper, but they fit
naturally with the viewpoint developed here.

\section*{Acknowledgements}

We thank Jay Pantone for running Kuszmaul's code to obtain the value $N_6=25$.  We also thank
Magn\'us M\'ar Halld\'orsson, Hjalti Magn\'usson, Joshua Sack, and Einar Steingr\'imsson for
discussions during the early stages of this work.  We further thank Magn\'us M\'ar Halld\'orsson
for helpful discussions relating to the computational testing of our conjectures.

Large language models were used in this work, more precisely ChatGPT-5.4 and Claude Opus 4.6.
They were used for finding examples, writing code, and shortening and tightening proofs,
which in some cases led to better arguments.  On the whole, we can recommend the use of
an AI partner in mathematics research, provided that one maintains tight control: these systems can
disappear into rabbit holes, start expensive computations and then forget about them, and produce
arguments that require careful untangling to eliminate circular reasoning. All errors are still
the sole responsibility of the human author.

\section{Computational reproducibility}\label{sec:repro}

The computer-assisted statements in the main body are all exact and exhaustive.  The scripts in
\texttt{code-permuta/} are available in the public repository
\url{https://github.com/ulfarsson/Game_patterns}.  They are limited to the computations needed for
the main paper, with the sole exception that the value $N_6=25$ was obtained separately and
is not part of the routine Python bundle.  The runs cited here were checked with Python 3.11 and
\texttt{permuta}~\cite{permuta-repo}.  After installing \texttt{permuta}, the commands below can
be run directly from the repository root.

To reproduce the starting-position data from Section~2, including Table~\ref{tab:sgvalues},
Proposition~\ref{prop:k4-reverse-data}, and the numbers in
Example~\ref{ex:n9-reverse-fails}, together with the optimal-play length distributions shown in
Figure~\ref{fig:k4-optimal-lengths}, run:
\begin{verbatim}
python code-permuta/pap_starting_data_permuta.py table
python code-permuta/pap_starting_data_permuta.py reverse-k4
python code-permuta/pap_starting_data_permuta.py \
  optimal-dist n 4
python code-permuta/pap_starting_data_permuta.py \
  query 9 4 1234,4321,1324
python code-permuta/pap_starting_data_permuta.py \
  query 9 4 1234,4321,1324,4231
\end{verbatim}
The two query commands produce the counts $334$ and $2$ from
Example~\ref{ex:n9-reverse-fails} together with the
winning replies in that state, printed in \texttt{permuta}'s native zero-based notation.  For each
$n \in \{5,6,7,8,9,10\}$, the \texttt{optimal-dist} command prints the number of optimal play
lines of each possible length from the starting position~$\Sym_n$.

To reproduce the $k=4$ computations from Section~6 that are used directly in the proof,
namely the residual $n=7$ support classification in
Proposition~\ref{prop:residual-nonhard-inflations} and the hard-pair computation in
Proposition~\ref{prop:hard-pair-families}, run:
\begin{verbatim}
python code-permuta/k4_reverse_reply_permuta.py all
\end{verbatim}
This script prints the summary lines
\begin{verbatim}
gap_cases ok checks=1156
hard_pair ok supports=138 legal_states=543
\end{verbatim}

To reproduce the threshold data from Proposition~\ref{prop:exact-thresholds} up to $k=5$, run:
\begin{verbatim}
python code-permuta/bk_thresholds_permuta.py 3 6
python code-permuta/bk_thresholds_permuta.py 4 10
python code-permuta/bk_thresholds_permuta.py 5 20
\end{verbatim}
These commands print the counts $|\Av_n(B_k)|$ and the first length at which only the two
monotone permutations remain, equivalently the threshold value $N_k$; for $k=3$ this threshold is
$1$ because $\id_1=\id_1^r$.  The value $N_6=25$ was obtained separately and is not part of
the routine \texttt{permuta} bundle.

To reproduce the $k=5$ full-space computations cited in the discussion after
Theorem~\ref{thm:k4} and in the conclusion, run:
\begin{verbatim}
python code-permuta/k5_fullspace_data_permuta.py all
\end{verbatim}
This prints the counts
\begin{verbatim}
one_companion_targets=28 total_nonB5_nonmonotone=110
full_space_separable_pairs=59 total_nonmonotone_pairs=59
\end{verbatim}
and also lists the corresponding targets and confirms that no reverse pair is
inseparable in $\Sym_8$.

\bibliographystyle{amsplain}
\bibliography{references}

\end{document}